\documentclass[12pt]{amsart}

\usepackage{amssymb}
\usepackage{latexsym}
\usepackage{amsmath}

\def\wt#1{{{\widetilde #1} }}

\def\wh#1{{{\widehat #1} }}
\def\ov{\overline}
 \def\up{\upharpoonright}
\def\cH{\mathcal H}
\def\cD{\mathcal D}
\def\cC {\mathcal C}

\def \cF {\mathcal F}
\def \cK {\mathcal K}
\def\cL {\mathcal L}
\def\cB{\mathcal B}

\def \gH{\mathfrak H}

\def \gN{\mathfrak N}
\def \dC{{\mathbb C}}

\def \bC{\mathbb C}

\def \bR{\mathbb R}

\def \C{\widetilde {\mathcal C}}

\def \cd {\cdot}

\def \l{\lambda}  \def \s{\sigma}   \def\g {\gamma}
\def\d {\delta}   
  \def \G{\Gamma} \def \S{\Sigma}

\def \exa { {\rm Ext}_A} \def \AD {A\in [\cD,\gH]}

\def \im{{\rm Im}\,} \def \re{{\rm Re}\,}

\def\dom {{\rm dom}\,}
\def\ran {{\rm ran}\,}
\def\ker {{\rm ker }\,}
\def\mul {{\rm mul}\,}
\def\cdom{{\rm {\overline{dom}}\,}}
\def\cran{{\rm {\overline{ran}}\,}}

\def\bt{\{\cH,\G_0 ,\G_1 \}}
\def\bta{\{\cH\oplus\cH,\G,\G^\top \}}
\def \DP {\{ A,A^\top \}} \def \DA {\{ A,A \}}

\newtheorem{theorem}{Theorem}[section]
\newtheorem{proposition}[theorem]{Proposition}
\newtheorem{corollary}[theorem]{Corollary}
\newtheorem{lemma}[theorem]{Lemma}

\newtheorem*{example}{Example}
\theoremstyle{definition}
\newtheorem {definition} [theorem]{Definition}
\theoremstyle{remark}
\newtheorem{remark}[theorem]{Remark}
\numberwithin{equation}{section}

\begin{document}
\title[Unitary equivalence of some classes of proper extensions]
{Unitary equivalence of  proper extensions of a symmetric operator and the Weyl
function}
\author {Seppo Hassi}
\address{Department of Mathematics and Statistics,
University of Vaasa \\
P.O. Box 700\\
 65101 Vaasa\\
 Finland}
\email{sha@uwasa.fi}

\author {Mark Malamud}

\address{Institute of Applied Mathematics and Mechanics, NAS of
Ukraine\\
 R. Luxemburg Str. 74\\
 83050 Donetsk\\
  Ukraine}
\email{mmm@telenet.dn.ua}

\author {Vadim Mogilevskii}
\address{Department of Mathematical Analysis\\
Lugans'k National University\\
 Oboronna Str. 2\\
 91011  Lugans'k\\
   Ukraine}
\email{vim@mail.dsip.net}



\begin{abstract}
Let $A$ be a densely defined simple symmetric operator in $\gH$, let $\Pi=\bt$
be a boundary triplet for $A^*$ and let $M(\cd)$ be the corresponding Weyl
function. It is known that the Weyl function $M(\cd)$ determines the boundary
triplet $\Pi$, in particular, the pair $\{A,A_0\}$, where $A_0:=
A^*\lceil\ker\G_0 (= A^*_0)$, uniquely up to unitary similarity. At the same
time the Weyl function corresponding to a boundary triplet for a dual pair of
operators defines it uniquely only up to weak similarity.

In this paper we consider symmetric dual pairs $\{A,A\}$ generated by $A\subset
A^*$ and special boundary triplets $\wt\Pi$ for $\{A,A\}$. We are interested
whether the result on unitary similarity remains valid provided that the Weyl
function corresponding to $\wt\Pi$ is $\wt M(z)= K^*(B-M(z))^{-1} K,$ where $B$
is some non-self-adjoint bounded operator in $\cH$. We specify some conditions
in terms of the operators $A_0$ and $A_B= A^*\lceil \ker(\G_1-B\G_0)$, which
determine uniquely (up to unitary equivalence) the pair $\{A,A_B\}$ by the Weyl
function $\wt M(\cd)$. Moreover, it is shown that under some additional
assumptions the Weyl function $M_\Pi(\cdot)$ of the boundary triplet $\Pi$ for
the dual pair $\DA$ determines the triplet $\Pi$ uniquely up to unitary
similarity. We obtain also some negative results demonstrating that in general
the Weyl function $\wt M(\cd)$ does not determine the operator $A_B$ even up to
similarity.
   \end{abstract}

\maketitle

\section{Introduction}
Let $\gH$ be a Hilbert space, let $A$  be a densely defined closed symmetric
operator in $\gH$ with equal deficiency indices $n_+(A)=n_-(A)\leq\infty$, and
let $A^*$ be the adjoint operator of $A$.

During the last three decades a new approach to the extension theory has been
elaborated and has already attracted some attention. It is based on a concept
of a boundary triplet $\Pi=\bt$ for the operator $A^*$ (see \cite{GG91}). The
main ingredient of this approach is the following abstract Green's (Lagrange)
identity
   \begin{equation}\label{Intro1}
(A^*f,g) - (f,A^*g) = (\G_1  f,\G_0  g)_{\cH} - (\G_0f,\G_1 g)_{\cH}, \quad f,
g\in \dom A^*,
   \end{equation}
where $\cH$ is an auxiliary Hilbert space and $\G_0,\;\G_1$ are linear mappings
from $A^*$ to $\cH$ such that the mapping $\G=(\G_0\;\;\G_1)^\top$ is
surjective. A boundary triplet for $A^*$ always exists but is not unique. Its
role in the extension theory is similar to that of a coordinate system in
analytic geometry. It allows one to parameterize the set $\exa$ of closed
extensions $\widetilde A$ of $A$ satisfying $A\subset \wt A \subset A^*$
(proper extensions) in terms of abstract boundary conditions. Namely, the
equality $\Theta =\G \dom \wt A$ establishes a one-to-one correspondence
between all extensions $A_\Theta := \wt A \in \exa$ and all closed linear
relations $\Theta$ in $\cH$.  If $\Theta$ is the graph of a closed operator
$B(\in \mathcal C(\cH))$, then the corresponding extension is given by
%
%
   \begin{equation}\label{Intro2}
\wt A =  A_B= A^*\lceil \ker(\G_1-B\G_0).
   \end{equation}
%
%
%
%

The main analytical tool in this approach is the Weyl function $M(\cdot)$
corresponding to $\Pi$ that was introduced and investigated extensively in
\cite{DM91}. It is defined by
%
%
  \begin {equation}\label{Intro3}
\G_1{f_z} = M(z)\G_0{f_z}, \quad f_z\in \gN_z:=\ker (A^*-z), \quad
z\in\bC\setminus\bR.
  \end{equation}
It is shown in \cite{DM91} that $M(\cd)$ is well defined, takes values in
$[\cH]$ and is an $R[\cH]$-function (Nevanlinna function), i.e., it  is
holomorphic in $\bC\setminus\bR$,  $\im z \im M(z)\geq 0$ and $M^*(z)=\ov {M
(z)}, \; z\in\bC\setminus\bR$.

If the operator $A$ is simple, then \emph{the Weyl function $M(\cdot)$
determines the pair  $\{A,A_0\},$ $A_0 = A^*\lceil \ker\Gamma_0$, as well as
the boundary triplet $\Pi$ itself, uniquely up to unitary equivalence}; see
\cite{DM91,Mal92}. In particular, $M(\cdot)$ determines the extension $A_B$
given by \eqref{Intro2} uniquely up to the unitary similarity. Note that the
unique  determination  of the pair $\{A,A_0\}$ up to the unitary equivalence
has been proved
in \cite{KrL73,LT77} in terms of the  so-called $Q$-functions.

The concept of a  boundary triplet for $A^*$ has been extended to the case of
dual pairs of closed operators $A,A^T$, i.e. pairs of operators satisfying
$A\subset (A^\top)^*$; the definition reads as follows:
   \begin{definition}\label{def1.1}$\,$\cite{L-S}
Let $\{A,A^T\}$ be a dual pair of closed densely defined operators $A$ and
$A^T$ in $\gH$.  A collection $\Pi= \{\cH_0\oplus\cH_1, \G, \G^\top \}$, where
$\cH_0$ and $\cH_1$ are Hilbert spaces and
\begin {equation*}
\G=\begin{pmatrix} \G_0\cr \G_1 \end{pmatrix}: \dom (A^\top)^*\to
\cH_0\oplus\cH_1, \quad \G^\top=\begin{pmatrix} \G_0^\top\cr \G_1^\top
\end{pmatrix}: \dom A^* \to \cH_1\oplus\cH_0
\end{equation*}
are linear mappings, is called a boundary triplet for $\DP$ if the mappings
$\G$ and $\G^\top$ are surjective and for every $f\in \dom(A^\top)^*$ and $g\in
\dom A^*$ the following abstract Green's identity holds
    \begin{equation*}
((A^\top)^*f, g)-(f, A^*g) = (\G_1 f, \G_0^\top g)-(\G_0 f, \G_1 ^\top g).
   \end{equation*}
\end{definition}


First constructions of boundary triplets for a dual pair of non-self-adjoint
elliptic operators as well as their applications to non-local boundary value
problems for elliptic operators in domains with smooth boundary go back to the
classical papers by M.I. Visik \cite{Vi52} and G. Grubb \cite{Gr68}.

In \cite{MalMog97, MalMog02} a concept of the Weyl function $M_\Pi(\cd)$
corresponding to the triplet $\Pi$ have also been extended to the case of dual
pairs $\{A,A^T\}$. The considerations in \cite{MalMog97, MalMog02} have been
inspired by investigations of V.B. Lidskii \cite{Lid60} in the spectral theory
of the Sturm-Liouville operator $-\frac{d^2}{dx^2}+q$ with complex-valued
potential; there the Weyl function was introduced by extending the Weyl
limit-circle procedure. It turned out that the abstract Weyl function
$M_\Pi(\cd)$ from \cite{MalMog97, MalMog02} coincides with that from
\cite{Lid60}. Moreover, it was shown in \cite{MalMog02} that for each boundary
triplet $\Pi$ for $\DP$ the extension $A_0 := (A^\top)^*\lceil \ker \G_0$ and
the Weyl function $M_\Pi(\cd)$ coincide, respectively, with the main operator
and the transfer function of some linear stationary system (in the sense of
\cite{Aro79,AroNud96}). Further investigations in this topic have been
motivated by possible applications to boundary value problems for
non-self-adjoint differentials operators (see \cite{MalMog02, BroGruWoo09,
BrHinMar09,BrMarNab08,Gru10, Mal10} and references therein).

It is shown in \cite{MalMog99} that in the case of a simple dual pair $\DP$ the
Weyl function $M_\Pi(\cd)$ determines the boundary triplet $\Pi$ as well as the
proper extension $A_B$ \emph{uniquely up to weak similarity} (similar result
for linear stationary systems was earlier obtained in \cite{Aro79}). Note that
weak similarity does not preserve the spectral properties and even the spectrum
of the main operator. In recent publications \cite{AroNud02} and
\cite{ArlHasSno05,ArlHasSno07} it was shown that in two special cases the Weyl
function determines the operator $A_B$ up to similarity \cite{AroNud02} and up
to unitary similarity \cite{ArlHasSno05,ArlHasSno07}, respectively.

In this paper we consider only special dual pairs $\{A,A\}$ generated by a
symmetric, not necessarily densely defined, operator $A$. Moreover, we consider
only special boundary triplets $\Pi= \{\cH\oplus\cH, \G, \G^\top \}$ for the
dual pair $\{A,A\}$ such that the corresponding Weyl function $M_\Pi(\cd)$ is
     \begin{equation}\label{Intro5}
\wt M(z) := M_\Pi(z) = K^*(B-M(z))^{-1} K, \quad z\in \rho (A_B),
     \end{equation}
where $ M(\cdot)$ is the Weyl function of $A$ defined by \eqref{Intro3} and $B$
is a bounded non-self-adjoint ($B\not =B^*$) operator in $\cH$. The latter
means, in particular, that $A_B$ is the almost solvable extension of $A$ in the
sense of \cite{DM92}.

The  Weyl function $M_\Pi(\cd)$ of the form \eqref{Intro5} is no longer
Nevanlinna function in $\Bbb C_{\pm}$. However, if $B$ is accumulative, then
the Weyl function $M_\Pi(\cd)$ of the form \eqref{Intro5} is  a Nevanlinna
function in $\Bbb C_+$, i.e. it is holomorphic in $\Bbb C_+$
 and $\im M_\Pi(z) \geq 0, \ z\in \Bbb C_+$. It follows that  $M_\Pi(\cd)$ admits a classical  integral representation (cf.  \eqref {2.2}) only in $\Bbb C_+$.

We are interested in sufficient (and necessary) conditions in terms of the
operators $A_B$ and $A_0=A^*_0$ that make it possible  to determine the pair
$\{A,A_B\}$ uniquely (up to the unitary equivalence) by  the Weyl function
\eqref{Intro5}. More precisely, we consider the following problem:

\emph{Given a simple symmetric operator $A^{(j)}$, a boundary triplet
$\Pi_j=\;\;\;$ $ \{\cH,\Gamma^{(j)}_0,\Gamma^{(j)}_1\}$ for $A^{(j)*}$, the
corresponding Weyl function $M_j(\cdot)$, a proper extension $A^{(j)}_{B_j} =
{A^{(j)*}}\lceil \ker(\G_1^{(j)} - B_j\G_0^{(j)}),\ j\in\{1,2\},$ and a domain
$\Omega \subset \rho(A^{(1)}_{B_1})\cap \rho(A^{(2)}_{B_2})\subset \Bbb C$.
When the equality}
     \begin {equation}\label{Intro6}
\wt M_1(z) = K^*_1(B_1 - M_1(z))^{-1}K_1  = K^*_2(B_2 - M_2(z))^{-1}K_2 = \wt
M_2(z), \quad z\in\Omega,
     \end{equation}
\emph{yields  the unitary  similarity  of the pairs of operators}
$\{A^{(1)}_{B_1}, A^{(1)}_0\}$ \emph{and}

\noindent $\{A^{(2)}_{B_2}, A^{(2)}_0\}?$

We show (cf. Theorem \ref{th5.4}) that the answer is positive at least in the
following two cases:

(i) $\Omega\cap\Bbb C_{\pm} \not = \emptyset$;

(ii) $\Omega \subset \bC_+$  
and the $ac$-part $E_j^{ac}(\cdot)$ of the spectral measure $E_j(\cdot)$ of
$A^{(j)}_0,\  j\in\{1,2\},$ is not equivalent to the Lebesgue measure.

In particular, both assumptions (i) and (ii) are satisfied provided that
$\Omega \cap \Bbb R  \not =  \emptyset$. It is emphasized however, that
condition (ii) is not necessary for the unitary equivalence of $A_{B_1}$ and
$A_{B_2}$. Moreover, \emph{the unitary equivalence might happen even if both
$E_1^{ac}(\cdot)$ and $E_2^{ac}(\cdot)$ are spectrally equivalent to the
Lebesgue measure} (see Remark \ref{rem5.2a}).

To include in our considerations  bounded operators $A_0 = A_0^*$ we consider
dual pairs $\{A,A\}$ with a bounded nondensely defined symmetric operator $A$.
We show that  the  Weyl function $M_\Pi(\cd)$ corresponding to a special
boundary triplet $\Pi=\bta $ for $\DA$ is
  \begin {equation}\label{1.8}
M_\Pi(z) = \cF + K^*(A_0-z)^{-1} K, \quad z\in\rho (A_0),
  \end{equation}
(c.f. \eqref{Intro5}), where $A_0 =\ker \G_0\in [\gH]$ and $K\in [\cH,\gH]$,
$\cF\in [\cH]$ are the operators defined in terms of the boundary
triplet $\Pi$. Since $A_0$ is bounded,    
$M_\Pi(\infty)=\cF$ and it follows from each of the assumptions (i) and (ii)
that $M_\Pi(\cdot)$ determines the pair $\{A_0,A_B\}$ uniquely up to the
unitary equivalence. In fact, in this case a stronger statement is valid:
\emph{the Weyl function $M_\Pi(\cdot)$ determines the boundary triplet $\Pi$
uniquely up to unitary equivalence} (see Theorem \ref{th5.10}). A similar
result is also valid for unbounded $A$ provided that $\rho(A_0)\cap\Bbb R\not
=\emptyset$  (cf. Theorem \ref{th5.11}). This result can be reformulated as
follows: \emph{the transfer function of a special linear stationary system
determines it uniquely up to unitary similarity} (see Remark \ref{rem5.13} for
details).

In the last section we present some negative results demonstrating that in
general the function $\wt M(\cd)$ of the form \eqref{Intro5} does not determine
the operator $A_B$ uniquely even up to the similarity. For instance, it is
shown that for any symmetric operator $A^{(1)}$, any Weyl function $M_1(\cdot)$
of $A^{(1)}$, and any accumulative $B_1\in[\cH]$ there exists a (non-unique)
simple symmetric operator $A^{(2)}$ and a (non-unique) dissipative operator
$B_2\in[\cH]$ such that equality \eqref{Intro6} holds with $K_1=K_2=I_{\cH}$
and $\Omega \subset \bC_+.$ At the same time the operators $A_{B_1}$ and
$A_{B_2}$ have different spectra, hence  cannot be similar.  Observe also that
without additional restrictions  \emph{the Weyl function $\wt M(\cd)$ does not
determine the  extension  $A_B$ uniquely  up to the unitary similarity even in
the case of the accumulative  $B$}  (see Remark \ref{rem5}(\ref{point4})).

 Finally, we present some explicit examples illustrating the above effect.

The main results of the paper have been announced without proofs in
\cite{HasMalMog12}.

\textbf{Notation.}  Throughout the paper $\gH$ and $\cH$ are assumed to be
separable Hilbert spaces, the set of bounded linear operators from $\cH_0$ to
$\cH_1$ is denoted by $[\cH_0,\cH_1]$, we also write $[\cH]:=[\cH,\cH]$.
Further, $P_\cL\in [\gH]$ denotes the orthoprojector in $\gH$ onto the subspace
$\cL\subset \gH$. The open upper and lower half-plane of $\Bbb C$ are denoted
by $\Bbb C_+$ and $\Bbb C_-$, respectively.

\section{Preliminaries}
\subsection{Linear relations}
Let $\cH_0$ and  $\cH_1$ be Hilbert spaces. A linear relation $T$ from $\cH_0$
to $\cH_1 $ is a linear manifold in $\cH_0 \oplus \cH_1$. In particular, the
set of all closed linear relations from $\cH_0$ to $ \cH_1$ is denoted by $\C
(\cH_0,\cH_1)$, and $\C(\cH) := \C (\cH,\cH)$. For each $T\in \C (\cH_0,\cH_1)$
the domain, the range, the kernel and the multi-valued part of $T$ are denoted
by $\dom T,\; \ran T,\; \ker T$, and $\mul T$, respectively. The notations
$\cdom T$ and $\cran T$ stand for the closures of the domain $\dom T$ and the
range $\ran T$ of $T$. Systematically a closed linear operator $T$ from $\cH_0$
to $\cH_1$ will be identified with its graph
\[
 {\rm gr}\, T=\{\{f, Tf\}: f\in \dom T\}\in \C(\cH_0,\cH_1).
\]
If $T\in \C(\cH_0,\cH_1)$ then the inverse linear relation $T^{-1}$ is given by
$$
T^{-1}=\{\{f^{\prime},f\}:\{f,f^{\prime}\}\in T\}\  (\in\C(\cH_1,\cH_0)).
$$
Furthermore, the adjoint linear relation $T^*$ is defined as
$$
T^*=\{\{g,g^{\prime}\}\in \cH_1\oplus \cH_0:(f^{\prime},g)=(f,g^ {\prime}),\ \
\{f,f^{\prime}\}\in T\}\in \C(\cH_1,\cH_0).
$$
For a linear relation $T\in \C (\cH_0,\cH_1)$ the following notations are
frequently used:

$0\in \rho (T)$~\ if~\ $\ker T=\{0\}$\  and\  $\ran T=\cH_1$, or equivalently,
if $T^{-1}\in [\cH_1,\cH_0]$;

$0\in \wh\rho (T)$~\ if~\  $\ker T=\{0\}$\ and\ $\cran T=\ran T\neq\cH_1;$

$0\in \sigma_c (T)$~\  if~\  $\ker T=\{0\}$ and $\cran T=\cH_1\neq\ran T;$

$0\in \sigma_p(T)$~\  if~\  $\ker T\neq\{0\};$

$0\in \sigma_r(T)$~\ if~\ $0\in \sigma (T)\setminus (\sigma_p(T)\cup\sigma
_c(T)).$

For any $T\in \C(\cH)$ denote by $\rho (T)=\{\l \in \bC:\ 0\in \rho (T-\l)\}$\
and $\widehat\rho (T)=\{\l \in \bC:\ 0\in \wh\rho (T-\l)\}$ the resolvent set
and the set of regular type points of $T$, respectively. The spectrum of $T$ is
given by $\sigma (T)=\bC \backslash \rho (T)$. It admits the following
classification:

$\sigma_c(T)=\{\l \in \bC:0\in \sigma_c (T-\l)\}$ is the continuous spectrum;

$\sigma_p (T)=\{\l \in \bC:0\in \sigma_p (T-\l)\}$ is the point spectrum;

$\sigma_r (T)=\{\l \in \bC:0\in \sigma_r (T-\l)\}$ is the residual spectrum.

A linear relation $T\in\C (\cH)$ is called symmetric  if $T\subset T^*$ and
self-adjoint  if $T=T^*.$

\subsection{Operator measures}
Here some known facts on operator measures are recalled.
Let ${\cH}$ be a separable Hilbert space, let $\cB(\bR)$ be the Borel
$\s$-algebra of the real line $\bR$ and let $\cB_b(\bR)$ be the ring of all
bounded sets in $\cB(\bR)$.
    \begin{definition}\label{def2.1}
(i)\  A mapping $\S(\cd):\cB_b(\bR)\to [\cH]$ is called an operator measure if
$\S(\emptyset)=0$, $\S(\d)=\S(\d)^*\geq 0 \; (\d\in\cB_b(\bR))$, and the
function $\S(\cd)$ is strongly countably additive.

(ii)\  An operator measure $\S(\cd)$ is called bounded if it is defined on
$\cB(\bR)$.

(iii)\  A bounded operator measure $\S(\cd)=:E(\cd)$ is said to be orthogonal
if $E(\bR)=I$ and  $E^2(\d)=E(\d)$ (i.e., $E(\d)$ is
the orthoprojector in $\cH$) . 
   \end{definition}
By the spectral theorem \cite{BirSol}, there exists a one-to-one correspondence
between the orthogonal measures $E(\cd)$ in $\cH$ and self-adjoint operators
$T=T^*$ in $\cH$. It is given by the decomposition
 \begin {equation}\label{2.0.1}
T=\int_\bR t\, dE(t),\quad   \dom T  = \{f\in \cH: \int_\bR
t^2d(E(t)f,f)<\infty \}.
  \end{equation}
The measure $E(\cd)$ is called the spectral measure of $T$.


The operator  measure $\S_1$ is called absolutely continuous with respect to
the measure $\S_2$ (in symbols: $\S_1\prec \S_2$) if $\S_2(\d)=0$ implies
$\S_1(\d)=0$ for $\d\in\cB_b(\bR)$. The measures $\S_1$ and $\S_2$ are called
equivalent ($\S_1\sim\S_2$) if $\S_1\prec \S_2$ and $\S_2\prec \S_1$.

In the sequel we denote by $m(\cd) $ the (scalar) Lebesgue measure in $\bR$.
The operator measure $\S(\cd)$ is called singular (with respect to $m(\cd)$) if
there exists a set $\d_0\in\cB(\bR)$ such that $m(\d_0)=0$ and
$\S(\d)=\S(\d\cap\d_0)$ for all $\d\in\cB_b(\bR)$. The singularity of $\S(\cd)$
is denoted by $\S\perp m$.

Each operator measure $\S(\cd)$ admits the Lebesgue  decomposition
  \begin {equation}\label{2.1}
\S=\S^{ac}+\S^s, \quad  \text{where} \quad \S^{ac}\prec m \quad\text{and} \quad
\S^s \perp m.
  \end{equation}
%
%
The operator measures $\S^{ac}$ and $\S^s$ are called the absolutely continuous
and singular parts of $\S,$ respectively. If the  measure $\Sigma(\cd)= E(\cd)$
is orthogonal, then its absolutely continuous and singular parts $E^{ac}(\cd)$
and $E^{s}(\cd)$ are also orthogonal, $E^{ac}(\cd)E^{s}(\cd)=0$, and the
decomposition in \eqref{2.1} can be rewritten as $E=E^{ac}\oplus E^s.$

For a self-adjoint operator $T$ in $\cH$ with the spectral measure $E(\cd)$
denote $\cH_{\tau}:=E^{\tau}\cH$ and $T_{\tau}=T\lceil\cH_{\tau}$, $\tau=ac,s$.
This yields the decompositions  $\cH=\cH_{ac}\oplus \cH_s$ and $T=T_{ac}\oplus
T_s$. The operators $T_{ac}$ and $T_s$ are called the $ac$-part  and the
singular part of $T$, respectively. The spectrum $\sigma(T_{ac})$
($\sigma(T_{s})$) is called the $ac$-spectrum (resp. the singular spectrum) of
$T.$

%
%
%
%
%

Recall that each self-adjoint relation $T\in\C (\cH)$ admits the decomposition
$$
T=T'\oplus \wh \mul T, \quad \wh\mul T=\{0\}\oplus \mul T,
$$
where $T'$ is a self-adjoint operator (the operator part of $T$) in $\cH'
:=\cH\ominus \mul T$.
    \begin{definition}\label{def2.1.2}
Let $T = T^*\in\C (\cH)$ and   $\mul T\neq \cH$, i.e., $\cH'\neq \{0\})$. Then
the spectral measure of $T$ is defined to be the spectral measure
$E(\cd):\cB(\bR)\to [\cH']$ of its operator part $T' = (T')^*$. Moreover, the
$ac$-spectrum $\s_{ac}(T)$ and singular spectrum $\s_{s}(T)$
 of $T$ are defined as 
 $\s_{ac}(T):=\s_{ac}(T')$ and $\s_{s}(T):=\s_{s}(T')$, respectively.
   \end{definition}

\subsection{$R$-Functions}
A holomorphic operator-valued function $F(\cdot):\bC\setminus\bR\to [\cH ]$ is
called an $R$-function (Nevanlinna function) if $\im z \cd \im F(z) \geq 0$ and
$F^*(z)=F(\overline z), \; z\in \bC\setminus\bR$.
The class of $R$-functions with values in $[\cH]$ is denoted by $R[\cH]$. Every
$F(\cdot)\in R [\cH]$ admits an integral representation of the form (see
\cite{Bro69})
  \begin {equation} \label{2.2}
F(z) = C + D z + \int_{\bR} \left(\frac{1}{t-z}- \frac{t}{1 + t^2}\right)
d\S,\quad z\in \bC\setminus\bR,
   \end{equation}
where $C,\;D\in [\cH],\;C = C^*, \ D\geq 0$ and $\S(\cd):\cB_b(\bR)\to [\cH]$
is an operator measure satisfying
\begin {equation} \label{2.2a}
\int_\bR  \frac{d \S } {1+t^2}  \in [\cH]
\end{equation}
(the integrals are understood in the strong sense).  The operator measure
$\S(\cd)$  and the operators\   $C, \;D$ in \eqref{2.2} are called the spectral
measure and parameters of  $F(\cd)$, respectively. They are uniquely defined by
$F(\cd)$. Moreover, a distribution operator function $\S(\cd):\bR\to [\cH]$
defined by
\begin {equation} \label{2.2b}
\S (t)=\left\{\begin{array}{cl}
   \S ([0,t)), & t>0, \cr
    0,         & t=0, \cr
  -\S([t,0)),  & t<0,
\end{array}
\right.
\end{equation}
is called the (normalized) spectral function of $F(\cd)$.

%
%
%
%

\subsection{Boundary triplets and the Weyl functions}
Let $\gH$ be a  Hilbert space, let $A$ be a closed symmetric operator in $\gH$,
not necessarily densely defined, and let $A^*(\in\C (\gH))$ be the adjoint of
$A$. Assume also that the operator $A$ has equal deficiency indices
$n_\pm(A)=\dim(\ker (A^*\mp i)) \leq \infty$.

Here we briefly recall the basic facts on  boundary triplets and the
corresponding Weyl functions following \cite{GG91,DM91,Mal92,DM95}.
\begin{definition}[\cite{GG91,Mal92}]\label{def2.2}
A triplet $\Pi = \{\cH, \G_0, \G_1\}$, where $\cH$ is a Hilbert space and
$\G_0,\G_1:\  A^*\rightarrow \cH$ are linear mappings,  is called a boundary
triplet for $A^*$ if the following "abstract Green's identity" holds
\begin{equation}\label{2.3}
(f',g) - (f,g') = (\G_1 \wh f,\G_0 \wh g)_{\cH} - (\G_0 \wh f,\G_1\wh g)_{\cH},
\;\;\;\, \wh f=\{f,f'\},\  \wh g=\{g,g'\}\in A^*
\end{equation}
and the mapping $\G:=(\Gamma_0\;\;\Gamma_1)^\top:  A^* \rightarrow \cH \oplus
\cH$ is surjective.
\end{definition}
Note that $A^*$ is a densely defined operator if and only if $\dom A$ is dense
in $\gH$. In this case the identity \eqref{2.3} can be rewritten in the
equivalent form
\begin{equation}\label{2.3Dense}
(A^*f,g) - (f,A^*g) = (\G_1  f,\G_0  g)_{\cH} - (\G_0f,\G_1 g)_{\cH}, \quad f,
g\in \dom A^*.
\end{equation}

A boundary triplet $\Pi=\{\cH,\G_0,\G_1\}$ for $A^*$ exists whenever $n_+(A) =
n_-(A)$. Moreover, one has $n_\pm(A) = \dim\cH$, $\ker\Gamma_j$ ($j=1,2$) is
selfadjoint, and $\ker\Gamma_0 \cap \ker\Gamma_1=A$. It is known that
$\Gamma_0, \Gamma_1\in [A^*, \cH]$, that is the operators $\G_0$ and $\G_1$ are
bounded with respect to the graph topology on $A^*$.
   \begin{definition}\label{def2.3}
A closed linear relation $\wt A\in\C (\gH)$ is called a proper extension of
$A$, if $A \subsetneqq \wt A \subsetneqq A^*$. The set of proper extensions
augmented by $A$ and $A^*$ is denoted by $\exa.$
     \end{definition}
Clearly, every self-adjoint extension $\wt A= {\wt A}^*$ is automatically
proper, i.e., $\wt A\in \exa$. Moreover, every closed dissipative extension
$\wt A$ is also proper (see \cite{Mal92}). With a boundary triplet $\Pi$ one
typically fixes two extensions $A_j:=\ker\G_j, \ j\in \{0,1\}$ of $A$, which
are self-adjoint in view of Proposition \ref{pr2.4} below. Conversely, for
every $A_0=A_0^*\in \exa$ there exists a (non-unique) boundary triplet
$\Pi=\{\cH,\G_0,\G_1\}$ for $A^*$ such that $A_0=\ker\G_0$.

Using the concept of a boundary triplet one can parameterize the set of all
proper extensions of $A$ by means of the set $\C(\cH)$ of closed linear
relations in $\cH$.
%
\begin{proposition}\label{pr2.4}
Let  $\Pi=\{\cH,\G_0,\G_1\}$  be a boundary triplet for  $A^*.$  Then the
mapping
\begin {equation}\label{bij}
\exa\ni\ \widetilde A \to  \Gamma \widetilde A =\{\{\Gamma_0 \wh f,\Gamma_1 \wh
f \} : \wh f\in \widetilde A \} =: \Theta \in \widetilde\cC(\cH)
\end{equation}
establishes  a bijective correspondence between the sets $\exa$ and
$\widetilde\cC(\cH)$. We put $A_\Theta :=\widetilde A$ where $\Theta$ is
defined by \eqref{bij}. Moreover, the following statements hold:

{\rm (i)} if $\Theta:=B\in \cC(\cH)$ is an operator, then \eqref{bij} takes the
form
   \begin{equation} \label{2.4}
 A_B= \ker (\G_1-B\G_0);
   \end{equation}

{\rm (ii)} the extension $A_\Theta\in\exa$ is $m$-dissipative,
$m$-accumulative,  self-adjoint, if and only if so is the corresponding linear
(boundary) relation  $\Theta\in \widetilde\cC(\cH)$.
    \end{proposition}

The linear relation $\Theta$ (the operator $B$) is called \emph{the boundary
relation (the boundary operator)} or the coordinate of the extension
$A_\Theta$. In particular,  $A_j:=\ker \G_j = A_{\Theta_j},\; j\in \{0,1\},$
where $\Theta_0:= \{0\} \times \cH$ and $\Theta_1 := \cH \times \{0\}$. Hence
$A_j= A_j^*$ since clearly $\Theta_j = \Theta^*_j$. In the sequel the extension
$A_0$ is usually regarded as a reference self-adjoint extension.

It is well known that Weyl functions play an important role in the direct and
inverse spectral theory of singular Sturm-Liouville operators. In the papers
\cite{DM91,Mal92} the concept of Weyl function was generalized to the case of
an arbitrary
symmetric operator $A$ with  $n_+(A) = n_-(A).$  
Some basic facts on the Weyl functions and $\gamma$-fields are now briefly
recalled.

Let  $\gN_z(A):=\ker (A^*-z)$ be the defect subspace of  $A$ and let $\wh \gN_z
(A):=\{\{f_z, z f_z\}:\,f_z\in\gN_z(A)\}$. Clearly, $\wh\gN_z (A)$ is a
(closed) subspace in $A^*$.  Denote by $\pi_1$ the orthoprojector in
$\gH\oplus\gH$ onto $\gH\oplus\{0\}$.
\begin{definition}[{\cite{DM91,Mal92}}]\label{Weylfunc}
 Let $\Pi=\{\cH,\G_0,\G_1\}$ be a boundary triplet  for $A^*.$
The operator functions $\gamma(\cdot): \rho(A_0)\rightarrow [\cH,\gH]$ and
$M(\cdot): \rho(A_0)\rightarrow  [\cH]$ defined  by
\begin {equation}\label{2.5}
\gamma(z):=\pi_1\bigl(\Gamma_0\!\upharpoonright\wh\gN_z(A)\bigr)^{-1}, \;\;\;
\G_1\up\wh\gN_z(A)= M(z)\G_0\up\wh\gN_z(A), \quad z\in\rho (A_0),
\end{equation}
 are called the {\em $\gamma$-field} and the
{\em Weyl function}, respectively, corresponding to  $\Pi.$
 \end{definition}

It is  shown in \cite{DM91,Mal92} that the operator functions $\g(\cd)$ and
$M(\cd)$ are well defined and holomorphic on $\rho
(A_0).$  
Moreover, $M(\cdot)\in R^u[\cH]$, i.e.,  $M(\cdot)\in R[\cH]$ and
$0\in \rho(\im(M(i)))$.   

\noindent
A symmetric operator $A$ in $\gH$ is called simple if there is no nontrivial
decomposition
$$
\gH=\gH_1\oplus\gH_2, \quad A=A_1\oplus A_2,
$$
where $\gH_1 \neq \{0\}$, $A_1 = A_1^*\in \cC(\gH_1)$  and $A_2$ is a symmetric
operator in $\gH_2$.
%
%
\begin{proposition}\label{pr2.4.1}$\,$ \cite{Mal92,DM95,MalNei02}
Let $A$ be a simple symmetric operator in $\gH$ and let $\Pi= \bt$ be a
boundary triplet for $A^*$ with $A_0=\ker \G_0$. Let $M(\cd)$ be the
corresponding Weyl function and let $\S(\cd)$ be the spectral measure of
$M(\cd)$. If $\mul A_0\neq\gH$, then the spectral measure $E(\cd)$ of $A_0$ and
the measure $\S(\cdot)$ are equivalent, $E\sim \S.$  In particular, $E^{ac}\sim
\S^{ac}$ and $E^{s}\sim \S^{s}$.
    \end{proposition}

A description of spectra of proper extensions $A_\Theta$ of $A$ in terms of
$\Theta$ and the corresponding Weyl function $M(\cd)$ is given as follows.
    \begin{proposition}\label{pr2.5}$\,$\cite{DM91, DM95}
Let  $\Pi = \{\cH, \G_0,\G_1\}$ be  a boundary triplet for $A^*$ with the Weyl
function $M(\cd)$ and let $\Theta\in \C(\cH)$. Then for all $\l\in \rho(A_0)$
the following equivalences hold:
\begin{eqnarray*}
& \l\in\rho(A_\Theta)\iff 0\in \rho(\Theta-M(\l)),  \\
& \l\in\sigma_j (A_\Theta)\iff 0\in \sigma_j(\Theta-M(\l)),\quad
 j\in \{p,c,r\}.
\end{eqnarray*}
    \end{proposition}

\subsection{Boundary triplets and Weyl functions for a dual pair $\DA$}
%
%
%
%
%
%

Clearly, every symmetric operator $A$ in $\gH$ generates a dual pair of the
form $\DA$, and vise versa. In this case the definition of a boundary triplet
for a general dual pair from \cite{L-S,MalMog02} (see also Definition
\ref{def1.1}) is simplified and reads as follows.
   \begin{definition}\label{def2.10}
Let $A$ be a closed symmetric operator in $\gH$ with equal deficiency indices
$n_+(A)=n_-(A)\le\infty.$
A collection $\Pi=\bta$, where $\cH$ is a Hilbert space and
\begin {equation*}
 \G=\begin{pmatrix}\G_0\cr\G_1\end{pmatrix} : A^*\to\cH\oplus\cH,
 \quad \G^\top=
 \begin{pmatrix}\G_0^\top\cr\G_1^\top\end{pmatrix} :A^*\to\cH\oplus\cH
\end{equation*}
are linear mappings, is called a boundary triplet for $\DA$ if the mappings
$\G$ and $\G^\top$ are surjective and the following abstract Green's identity
holds
  \begin {equation}\label{2.11}
(f',g)-(f,g')=(\G_1 \wh f,\G ^\top_0 \wh g)- (\G_0\wh f,\G^\top_1 \wh g), \quad
\wh f=\{f,f'\}, \  \wh g=\{g,g'\}\in  A^*.
   \end{equation}
  \end{definition}
In the sequel we consider only dual pairs $\DA$ and specify for this case some
results from \cite{L-S,MalMog02,MalMog05} on boundary triplets of general dual
pairs $\DP$. Each boundary triplet $\Pi= \bta$ for $\DA$ satisfies the
relations
$$
\dim\cH=n_\pm(A) \quad \text{and}\quad
\ker\G_0\cap\ker\G_1=\ker\G_0^\top\cap\ker\G_1^\top=A.
$$
With a boundary triplet $\Pi= \bta$ for $\DA$ one associates two extensions
$A_j=\ker\G_j(\in\exa), \; j\in \{0,1\}$, which in this case are not
necessarily self-adjoint. In what follows it is always assumed that $\rho
(A_0)\neq \emptyset$.

The following result is a counterpart of Proposition \ref{pr2.4} for dual pairs
$\DA$.
       \begin{proposition}\label{pr2.11}
Let $\Pi=\bta$ be a boundary triplet for $\DA$. Then the mapping \eqref{bij}
establishes a bijective correspondence between the sets $\exa$ and $\C(\cH)$.
Moreover, with $\Theta:=B\in \cC(\cH)$ the formula \eqref{bij} takes the form
\eqref{2.4}.
  \end{proposition}
We usually  indicate the correspondence in Proposition \ref{pr2.11} as $\wt
A=A_\Theta$. The linear relation $\Theta$ is called the boundary relation or
the coordinate, of the proper extension $\wt A=A_\Theta$ in the triplet $\Pi$.
  \begin{definition}\label{def2.12}
Let $\Pi=\bta$  be a boundary triplet for  $\DA$. The operator functions
$\g_{\Pi} (\cd) :\rho(A_0)\to [\cH, \gH]$ and $M_\Pi (\cdot):\rho (A_0)\to
[\cH] $ defined by 
\begin {equation}\label{2.12}
\gamma_{\Pi}(z):= \pi_1(\Gamma_0\up  \wh {\gN}_z (A))^{-1},  \quad
\G_1\up\wh\gN_z(A)= M_\Pi(z)\G_0\up\wh\gN_z(A)
\end{equation}
with $z\in \rho (A_0)\ne\emptyset$ are called, respectively, the $\g$-field and
the Weyl function corresponding to $\Pi$.
  \end{definition}
In other words the Weyl function $M_\Pi(\cdot)$ is defined as follows (cf.
\eqref{2.5})
  \begin {equation}\label{1.6}
\G_1\{f_z, z f_z\}=M_\Pi(z)\G_0\{f_z, z f_z\}, \quad f_z\in \ker (A^*-z), \quad
z\in\rho (A_0).
  \end{equation}
The functions $\g_\Pi(\cd)$ and $M_\Pi(\cd)$ are well defined and holomorphic
on $\rho (A_0)$.

Each boundary triplet $\Pi=\bta$ generates a (reversed) boundary triplet
$\Pi_\top=\{\cH\oplus\cH, \G_\top, (\G_\top)^\top\}$ for $\DA$, which is
defined by
$$
\G_\top=\G^\top=\begin{pmatrix}\G_0^\top\cr\G_1^\top\end{pmatrix} :
A^*\to\cH\oplus\cH \;\;\; \text{and}\;\;\; (\G_\top)^\top=\G=
\begin{pmatrix}\G_0\cr\G_1\end{pmatrix} :A^*\to\cH\oplus\cH.
$$
For this triplet one has $A_{0\top}(=\ker \G_0^\top)=A_{0}^*$ and the
corresponding $\g$-field $\g_{\Pi_\top}(\cd): \rho(A_0^*)\to [\cH, \gH]$ and
the Weyl function $M_{\Pi_\top} (\cdot):\rho (A_0^*)\to [\cH] $ are  defined
for every $z\in \rho (A_0^*)$ by
  \begin {equation*}
\gamma_{\Pi_\top}(z):= \pi_1(\Gamma_0^\top\up  \wh {\gN}_z (A))^{-1},  \quad
\G_1^\top\up\wh\gN_z(A)= M_{\Pi_\top}(z)\G_0^\top\up\wh\gN_z(A).
  \end{equation*}
\begin{remark}\label{rem2.15}
In the case that $\G^\top=\G=\begin{pmatrix} \G_0\cr\G_1\end{pmatrix}$ the
boundary triplet $\Pi=\bta$ for the dual pair $\DA$ turns into the boundary
triplet $\Pi=\bt$ for $A^*$ in Definition~\ref{def2.2}, while the Weyl function
$M_\Pi(\cd)$ given by \eqref{2.12} becomes the Weyl function $M(\cd)$ defined
in \eqref{2.5}. This remark shows that a boundary triplet $\bt$ for $A^*$ in
the sense of Definition \ref{def2.2} can be regarded as a particular case of a
boundary triplet $\bta$ for $\DA$ in the sense of Definition \ref{def2.10}. To
distinguish between  two kinds of boundary triplets,  the triplet $\Pi=\bt$ for
$A^*$ will sometimes be called an ordinary boundary triplet.
   \end{remark}

\subsection{Unitary equivalent boundary triplets}
Let $\gH_1$ and $\gH_2$ be Hilbert spaces. For each unitary operator $U\in
[\gH_1, \gH_2]$ from $\gH_1$ onto $\gH_2$ denote $\wt U :=U\oplus U\in
[\gH_1\oplus\gH_1, \gH_2\oplus\gH_2]$.
%
%
   \begin{definition}\label{def2.5a}
The linear relations $T_1 \in \C(\gH_1)$ and $T_2 \in \C(\gH_2)$ are said to be
unitarily equivalent by means of the unitary operator $U\in [\gH_1, \gH_2]$, if
$$T_2=\wt U T_1\,\left(=\{\{Uf,Uf'\}:\, \{f,f'\}\in T_1\}\right). $$
   \end{definition}
%
        \begin{definition}\label{def2.6}
Let $A^{(j)}$ be a closed symmetric operator in $\gH_j$, $j\in\{1,2\}$, and let
$\Pi_j=\{\cH\oplus\cH,\G^{(j)},\G^{\top (j)}\}$ be a boundary triplet for
$\{A^{(j)},A^{(j)}\}$. The boundary triplets $\Pi_1$ and $\Pi_2$ are said to be
unitarily equivalent (unitarily similar) by means of the unitary operator $U\in
[ \gH_1, \gH_2]$, if
\begin {equation}\label{2.8}
 \wt UA^{(1)*}=A^{(2)*}  \quad
 \text{and}\quad
 \G^{(2)}\wt U \up  A^{(1)*}=\G^{(1)},
\end{equation}
The boundary triplets $\Pi_1$ and $\Pi_2$ are said to be unitarily similar if
they are unitarily similar by means of some unitary operator $U\in [\gH_1,
\gH_2]$.
  \end{definition}
The following lemma is immediate from Definition \ref{def2.6}.
   \begin{lemma}\label{lem2.7}
Let the  boundary triplets $\Pi_1$ and $\Pi_2$ for $\{A^{(1)},A^{(1)} \}$ and

\noindent $\{A^{(2)},A^{(2)} \}$, respectively, be unitarily similar by means
of $U.$ Then for any $\Theta\in\C (\cH)$ the corresponding proper extensions
$A_\Theta^{(1)}\in Ext_{A^{(1)}}$  and $A_\Theta^{(2)}\in Ext_{A^{(2)}}$ are
also unitarily equivalent by means of the same $U$. In particular, for any
$B\in [\cH]$ the extensions $A_B^{(j)}=\ker (\G_1^{(j)}-B \G_0^{(j)}), \; j\in
\{1,2\},$ are unitarily similar by means of $U.$
   \end{lemma}

Note that in the case $\G^{(j)}=\G^{\top (j)}, \;j\in \{1,2\},$ Definition
\ref{def2.6} coincides with the usual definition of unitary equivalence of
ordinary boundary triplets $\Pi_j=\{\cH, \G_0^{(j)}, \G_1^{(j)}\}$ for
$A^{(j)*}$. Moreover, in this case the following theorem holds.

\begin{theorem}$\!$\cite{DM91,DM95}\label{th2.8}
Let $A^{(j)}$ be a simple symmetric operator in $\gH_j$, let
$\Pi_j=\{\cH,\Gamma_0^{(j)},\Gamma_1^{(j)}\}$ be a boundary triplet for
$A^{(j)*}$ and let $M_j(\cd)$ be the corresponding Weyl function,
$j\in\{1,2\}$. Then the boundary triplets $\Pi_1$ and $\Pi_2$ are unitarily
equivalent if and only if
\begin {equation*}
 M_1(z) = M_2(z), \quad z\in \Bbb C_+.
 \end{equation*}
\end{theorem}

\section{Unitary equivalence of proper extensions}

\subsection{Basic lemma} 
Let $A$ be a closed symmetric, not necessarily densely defined, operator in
$\gH$ with equal deficiency indices $n_+(A)=n_-(A)\leq\infty$. If $A$ is
simple, then according to Theorem \ref{th2.8} and Lemma \ref{lem2.7}, the Weyl
function $M(\cd)$ of an ordinary boundary triplet $\Pi=\{\cH,\G_0,\G_1\}$ for
$A^*$ determines the relations $A,A^*,A_0,A_1$ uniquely up to unitary
equivalence.

In what follows we consider the following transform
\begin{equation}\label{5.1}
 \wt M(z):=K^*(B-M(z))^ {-1}K
\end{equation}
of the Weyl function $M(\cdot)$ with bounded, not necessarily self-adjoint,
operator $B$. We investigate whether the function $\wt M(\cd)$  determines the
operator $A$ and its proper extensions uniquely up to unitary equivalence. It
will be shown that in general this is not the case, although it is really true
under some additional assumptions on the relation $A_0= \ker \G_0$.

We start with the following basic lemma.

   \begin{lemma}\label{BasicLemma}
Let  $F_j(\cdot)\in R[\mathcal H_j]$ be an $R$-function and let
$\Sigma_j(\cdot),C_j,D_j$\ $j\in \{1,2\},$ be its spectral measure and the
parameters, respectively (see \eqref{2.2}), i.e.,
\begin {equation}\label{5.3}
 F_j(z)=C_j+D_j z+\int_{\bR} \left(\frac{1}{t-z}-\frac{t}{1+t^2}\right)\,d\Sigma_j,
 \quad j\in \{1,2\}.
\end{equation}
Moreover, let $\cH$ be a Hilbert space, let $B_j\in[\mathcal H_j]$ and $K_j\in
[\cH,\cH_j]$, $j\in \{1,2\}$, be operators with $0\in\rho(K_1)\cap\rho(K_2)$
and such that the set
    \begin {equation}
\Omega_+:=\{\,z\in\Bbb C_+:\
0\in\rho\bigl(B_1-F_1(z)\bigr)\cap\rho\bigl(B_2-F_2(z)\bigr)\}
    \end{equation}
is not empty, and let the operator measure $\wt\Sigma_j(\cd):\cB_b(\bR)\to
[\cH]$ and the operators $\wt C_j,\;\wt D_j,\;\wt B_j\in [\cH]$ be given by
 \begin {gather}
\wt\Sigma_j (\d):=K_j^{-1}\Sigma_j (\d)K_j^{-1*}, \quad \d\in\cB_b(\bR), \label{5.5A}\\
\wt C_j :=K_j^{-1}C_jK_j^{-1*}, \;\;\;\wt D_j :=K_j^{-1}D_jK_j^{-1*}, \;\;\;\wt
B_j:= K_j^{-1}B_jK_j^{-1*}. \label{5.5B}
       \end{gather}
Then the equality
  \begin{equation}\label{5.4}
K_1^*(B_1-F_1(z))^{-1}K_1=K_2^*(B_2-F_2(z))^{-1}K_2, \quad z\in \Omega_+,
   \end{equation}
implies the following equalities:
\begin{gather}
\wt\Sigma_1^{s}(\d) = \wt\Sigma_2^{s}(\d),\quad
 \wt\Sigma_2^{ac}(\d)-\wt\Sigma_1^{ac}(\d)
 =\frac {m(\d)}{\pi}(\im\wt B_2 - \im \wt B_1),\quad \d\in\cB_b(\bR),  \label{5.5.1} \\
 \wt C_1 - \re\wt B_1=\wt C_2- \re \wt B_2, \quad \wt D_1 = \wt D_2.     \label{5.5}
    \end {gather}
     \end{lemma}
               \begin{proof}
Starting with \eqref{5.4} and taking inverses we get
  \begin{equation}\label{5.6}
K_1^{-1}(B_1- F_1(z))(K_1^*)^{-1} = K_2^{-1}(B_2- F_2(z))(K_2^*)^{-1}, \quad
z\in \Omega_+.
  \end{equation}
Since $F_1(\cd)$ and $F_2(\cd)$ are Nevanlinna functions, the equality
\eqref{5.6} remains valid for all $z\in \bC_+$ by continuity. Substituting the
integral representations \eqref{5.3} in \eqref{5.6} one obtains
%
\begin{equation}\label{5.7}
   \begin{split}
& \wt C_1-\wt B_1+\wt D_1 z+\int_{\bR}
 \left (\frac{1}{t-z}-\frac{t}{1+t^2}\right )\,d\wt\Sigma_1 \\ 
&=\wt C_2-\wt B_2+\wt D_2 z+\int_{\bR} \left
(\frac{1}{t-z}-\frac{t}{1+t^2}\right )\,d\wt\Sigma_2 , \quad z\in\dC_+.
   \end{split}
\end{equation}
Dividing both sides of \eqref{5.7} by $z$ and passing to the limit as
$z\to\infty$ gives
  \begin {equation}\label{5.8}
 \wt D_1= \wt D_2.
  \end{equation}
Now by taking $z=i$ in \eqref{5.7} leads to
\begin {equation}\label{5.8a}
\wt C_1-\wt B_1+i\int_{\bR}\frac{d\wt\Sigma_1 }{1+t^2}
 =\wt C_2-\wt B_2+i\int_{\bR}\frac{d\wt\Sigma_2 }{1+t^2}.
\end{equation}
Since $\wt C_i=\wt C_j^*, \; j\in \{1,2\},$ the equality \eqref{5.8a} gives $
\wt C_1 - \re \wt B_1=\wt C_2- \re \wt B_2$.

On the other hand, by taking the imaginary parts on both sides of \eqref{5.7}
and subtracting the terms in \eqref{5.8} one obtains for all $y>0$ and
$x\in\bR$,
 \begin{gather}\label{5.9}
-\im \wt B_1+\int_{\bR}\frac{y}{(t-x)^2+y^2}\,d\wt\Sigma_1  = -\im \wt
B_2+\int_{\bR}\frac{y}{(t-x)^2+y^2}\,d\wt\Sigma_2.
   \end{gather}
For a given $h\in\cH$ denote by $\mu_{j,h}$ the scalar measure
\begin{equation}\label{5.9.1}
\mu_{j,h}(\d)=(\wt\Sigma_j(\d)h, h), \quad \d\in\cB_b(\bR), \quad j\in\{1,2\},
\end{equation}
and let $ \mu_{j,h}(t)$ be the corresponding scalar distribution function (see
\eqref{2.2b}). Since $\wt\Sigma_j $ satisfies \eqref{2.2a}, it follows that
\begin {equation}\label{5.9.2}
\int_\bR \frac {d \mu_{j,h}} {1+t^2}<\infty, \quad h\in\cH, \quad j\in\{1,2\}.
\end{equation}
Moreover, \eqref{5.9} implies that for every $h\in\cH$, $y>0$ and $x\in\bR$,
\begin {equation}\label{5.9.3}
-(\im \wt B_1 h,h)+\int_{\bR}\frac{y}{(t-x)^2+y^2}\,d\mu_{1,h}  = -(\im \wt B_2
h,h)+\int_{\bR}\frac{y}{(t-x)^2+y^2}\,d\mu_{2,h}.
\end{equation}
Let $X_{j,h}$ be the set of all $x\in\bR$ for which the derivative $\frac
{d\mu_{j,h}(x)}{dx}$ exists, $j\in\{1,2\}$. Since the measure $\mu_{j,h}$
satisfies \eqref{5.9.2}, we can pass to the limit in \eqref{5.9.3} as
$y\downarrow 0$.  Applying the Fatou theorem, we arrive at the basic equality
  \begin {equation}\label{5.9.4}
-(\im \wt B_1 h,h)+\pi \frac{d\mu_{1,h}(t)} {d t}=-(\im \wt B_2 h,h)+\pi
\frac{d\mu_{2,h}(t)} {d t},
 \end{equation}
which holds for every $h\in\cH$ and $t\in X_{1,h}\cap X_{2,h}$. Let
$\frac{d\mu_{j,h}^{ac}} {d m}(\cd)$ be the derivative of the (absolutely
continuous) measure $\mu_{j,h}^{ac}$ with respect to the Lebesgue measure $m$.
Since
$$
\frac{d\mu_{j,h}^{ac}} {d m}(t)=\frac{d\mu_{j,h}(t)} {d t} \;\; \text{for
a.e.}\;\; t\in\bR, \quad j\in \{1,2\},
$$
it follows from \eqref{5.9.4} that
  \begin {equation}\label{5.9.5}
\mu_{2,h}^{ac}(\d)-\mu_{1,h}^{ac}(\d)=\frac {m(\d)}{\pi} ((\im \wt B_2-\im \wt
B_1)h,h), \quad \d\in \cB_b(\bR), \quad h\in \cH.
  \end{equation}
Combining \eqref{5.9.5} with \eqref{5.9.1} one obtains the second equality in
\eqref{5.5.1}. Substituting this equality in \eqref{5.9} and using the Lebesgue
decomposition $\S_j(\cd)=\S_j^{ac}(\cd)+\S_j^{s}(\cd)$ (see \eqref{2.1}) leads
to
\begin {equation}\label{5.12}
\int_{\bR}\frac{y}{(t-x)^2+y^2}\,d\wt\Sigma_1^s
=\int_{\bR}\frac{y}{(t-x)^2+y^2}\,d\wt\Sigma_2^s.
\end{equation}
Since the operator-valued  measure is uniquely recovered  by its Poisson
transform (for instance, by means of  the Stielties inversion formula), the
first equality in \eqref{5.5.1} follows.
    \end{proof}

  \begin{corollary}\label{cor5.0}
Let the conditions of Lemma \ref{BasicLemma} be satisfied and let $\S_j(\cd)$
be the spectral function of $F_j(\cd)$ (see \eqref{2.2b}). Assume, in addition,
that at least one of the following conditions is fulfilled:

(i) The operator measure $\S_j^{ac}(\cd),\ j\in \{1,2\},$ is not equivalent to
the Lebesgue measure $m(\cd)$ and  equality \eqref{5.4} holds.

(ii) The set
   \begin{equation*}
\Omega_-:=\{\,z\in\Bbb C_-:\
0\in\rho\bigl(B_1-F_1(z)\bigr)\cap\rho\bigl(B_2-F_2(z)\bigr)\}
    \end{equation*}
is not empty and the following equality holds
\begin {equation}\label{5.15}
K_1^*(B_1-F_1(z))^{-1}K_1=K_2^*(B_2-F_2(z))^{-1}K_2, \quad
z\in\Omega:=\Omega_+\cup\Omega_-.
   \end{equation}

(iii) Equality \eqref{5.4} holds  and  for some $t_0\in\Bbb R$ the weak
derivatives  $w$-$\frac{d\Sigma_1}{dt}(t_0)$, $w$-$\frac{d\Sigma_2}{dt}(t_0)$
exist and
    \begin{equation}\label{3.23}
w\text{-}\frac{d\Sigma_1}{dt}(t_0) = w\text{-}\frac{d\Sigma_2}{dt}(t_0).
    \end{equation}
 Then the following  relations are valid
\begin{gather}
\wt\Sigma_1 = \wt\Sigma_2,\quad \im \wt B_1=\im \wt
B_2,\label{5.16}\\
\wt C_1 - \re\wt B_1=\wt C_2- \re \wt B_2, \quad \wt D_1 = \wt
D_2.\label{5.16a}
\end{gather}
     \end{corollary}
  \begin{proof}
First observe that according to Lemma \ref{BasicLemma} each of the assumptions
(i), (ii), (iii) imply  relations \eqref{5.5.1} and \eqref{5.5}.

 (i) Since the measure $\Sigma_j^{ac}$ is
not equivalent to the Lebesgue measure $m(\cd)$, there exists a bounded Borel
set $\delta_j$ of positive Lebesgue measure $m (\delta_j)>0$, such that
$\Sigma_j^{ac}(\delta_j)=0$; $j\in\{1,2\}$. Substituting these sets  in the
second equality in  \eqref{5.5.1} we get
  \begin{equation*}
\im \wt B_2-\im\wt B_1=-\frac \pi {m (\d_2)}\wt\S_1^{ac}(\d_2)\leq 0,\quad
 \im \wt B_2-\im\wt B_1=\frac \pi {m
(\d_1)}\wt\S_2^{ac}(\d_1)\geq 0.
   \end{equation*}
Hence  $\im \wt B_2=\im \wt B_1$, and the second equality in \eqref{5.5.1}
yields  $\wt\Sigma_1^{ac} = \wt\Sigma_2^{ac}$. Combining this equality with the
first equality in \eqref{5.5.1} yields  $\wt\Sigma_1 = \wt\Sigma_2$.

  (ii)  Starting with equality \eqref{5.15} for
$z\in\Omega_-$ and repeating the reasonings of the proof of Lemma
\ref{BasicLemma}, we arrive at the equality
\begin{gather*}
-\im \wt B_1+\int_{\bR}\frac{y}{(t-x)^2+y^2}\,d\wt\Sigma_1  = -\im \wt
B_2+\int_{\bR}\frac{y}{(t-x)^2+y^2}\,d\wt\Sigma_2,
   \end{gather*}
with  $y<0$ and  $x\in\bR$.  Therefore for any $y>0$ and $x\in\bR$ one has
   \begin{gather*}
-\im \wt B_1-\int_{\bR}\frac{y}{(t-x)^2+y^2}\,d\wt\Sigma_1 = -\im \wt
B_2-\int_{\bR}\frac{y}{(t-x)^2+y^2}\,d\wt\Sigma_2.
   \end{gather*}
Combining this equality  with \eqref{5.9} and using the uniqueness of the
Poisson transform, we arrive at  \eqref{5.16}.

(iii) As it was shown in the proof of Lemma \ref{BasicLemma}  equality
 \eqref{5.4} implies the basic equality \eqref{5.9.4}. Since $t_0\in X_{1,h}
 \cap X_{2,h}$ for all $h\in\cH$, we get  from \eqref{5.9.4}
   \begin {equation}\label{3.24}
-(\im \wt B_1 h,h)+\pi \frac{d\mu_{1,h}(t_0)} {d t}=-(\im \wt B_2 h,h)+\pi
\frac{d\mu_{2,h}(t_0)} {d t},\quad h\in\cH.
 \end{equation}
Combining \eqref{3.24} with \eqref{3.23} one obtains $\im \wt B_1=\im \wt B_2$.
Inserting   this equality in \eqref{5.5.1}, yields  the first equality in
\eqref{5.16}.
    \end{proof}
%
      \begin{remark}\label{rem5.2a}
Corollary \ref{cor5.0}(iii) demonstrates at the same time  that the condition
(i) of this corollary  is not necessary for the validity of the statement. In
fact, the conclusion of Corollary \ref{cor5.0} substantially depends on the
measures $\Sigma_j$ themselves, not only on their spectral types. We emphasize
that the assumptions (iii) of Corollary \ref{cor5.0}, hence the conclusions,
can be satisfied even if the operator measures $\Sigma^{ac}_j,\ j\in\{1,2\},$
are spectrally equivalent to the operator Lebesgue measure $m_{\cH} :=
I_{\cH}m$ in the sense of \cite[Definition 4.8]{MalMal03}.

Indeed, let $t_0$ be a common point for which the weak derivatives
$\frac{d\Sigma_j}{dt}(t_0)$, $j\in \{1,2\}$, exist. Choose a scalar function
$\varphi(\cdot)\in C^1(\Bbb R)$ satisfying
    \begin{equation*}
\varphi(t_0)=0,  \quad  \varphi(t)>0, \quad t\in\Bbb R\setminus\{t_0\} \quad
\text{and}\quad  \varphi(t)=1,\quad t\in\Bbb R\setminus (t_0-1, t_0+1)
    \end{equation*}
and define the measure $\Sigma_{j,\varphi}$ by setting
     \begin{equation}\label{3.25}
\Sigma_{j,\varphi}(\delta) := \int_{\delta}\varphi(t)d\Sigma_j(t), \quad \delta
\in \cB_b(\bR),\quad  j\in\{1,2\} .
    \end{equation}
Then the weak derivative  $w$-$\frac{d\Sigma_j,\varphi}{dt}(t_0),\ j\in
\{1,2\},$ exists and equals zero since for every $h\in\cH$
   \begin{equation}\label{3.26}
\frac{\bigl(d\Sigma_{j,\varphi}(t)h,h\bigr)}{dt}|_{t=t_0}=
\varphi(t_0)\frac{\bigl(d\Sigma_{j}(t)h,h\bigr)}{dt}|_{t=t_0} = 0,\quad j\in
\{1,2\}.
   \end{equation}
%
%

If $\Sigma^{ac}_j,\ j\in\{1,2\},$ are spectrally equivalent to the measure
$m_{\cH}$, then clearly $\Sigma_{j,\varphi}^{ac}\sim m_{\cH}$.  Moreover, their
multiplicity functions coincide,
   \begin{equation}
N_{\Sigma_{j,\varphi}^{ac}}(t) = N_{m_{\cH}}(t)=\dim\cH, \quad t\in\Bbb
R\setminus\{t_0\}.
  \end{equation}
Hence the operator measures $\Sigma_{j,\varphi}^{ac}$ and $m_{\cH}$ are also
spectrally equivalent.

Now let $F_j$, $j\in\{1,2\},$ be $R[\cH]$-functions as in \eqref{5.3} and let
$\Sigma_j$ be the spectral measure of $F_j$, $j\in\{1,2\}$. Let $K_1=K_2$ and
let ${B}_j\in [\cH]$, $j\in\{1,2\}$, be such that the equation \eqref{5.4} is
satisfied, while $F_1\not = F_2$; for instance, one can take $F_2(z)=F_1(z) \pm
i I$ with $z\in\dC_\pm$ and put $B_1=0$, $B_2=iI$. Here the operator measures
$\Sigma^{ac}_j,\ j\in\{1,2\}$ can be taken to be spectrally equivalent to the
measure $m_{\cH}$. In this case, the assumptions and the conclusions of
Corollary \ref{cor5.0} fail to hold, while the assumptions in
Lemma~\ref{BasicLemma} are fulfilled.

We now slightly modify the previous situation by replacing the operator
measures $\Sigma_j$ by $\Sigma_{j,\varphi}$ defined in \eqref{3.25} and define
the $R[\cH]$-functions $F_{j,\varphi}$ according to \eqref{5.3} with
$\Sigma_{j,\varphi}$ instead of $\Sigma_j,\ j\in\{1,2\}$. If now for some $\wt
B_j\in[\cH]$, $j\in\{1,2\}$, the equality \eqref{5.4} with $F_{j,\varphi}$
(instead of $F_j$, $B_j$) holds, then due to \eqref{3.26}, and in contrary to
the previous case, we have
 $F_{1,\varphi}\equiv F_{2,\varphi}$ and
 $\Sigma_{1,\varphi}(t) = \Sigma_{2,\varphi}(t),\ t\in{\Bbb R}$.
     \end{remark}

\subsection{Sufficient conditions for unitary similarity of ordinary boundary triplets}
The results in the previous subsection are now applied to establish the unitary
equivalence of proper extensions of $A$ as well as the unitary equivalence of
appropriate boundary triplets for $A^*$.

Combining  Lemma \ref{BasicLemma} with Proposition  \ref {pr2.5} yields the
following statement.
    \begin{proposition}\label{pr5.1}
Let  $A^{(j)}$ be a simple  symmetric operator in the separable Hilbert space
$\gH_j$, let $\Pi_j=\{\cH_j,\G_0^{(j)},\G_1^{(j)} \}$ be  a boundary triplet
for $A^{(j)*}$ and let $M_j(\cd)$ be the corresponding Weyl function with the
integral representation
   \begin {equation}\label{5.3A}
M_j(z)=C_j+D_j z+\int_{\bR} \left
(\frac{1}{t-z}-\frac{t}{1+t^2}\right)\,d\Sigma_j ,\quad j\in \{1,2\},
    \end{equation}
%
%
(see \eqref{2.2}). Moreover, let $\cH$ be a Hilbert space and let $B_j\in
[\cH_j]$ and $K_j\in [\cH,\cH_j],\; j\in \{1,2\},$ be operators such that
$0\in\rho(K_1)\cap\rho(K_2)$ and
   \begin {equation*}
 \Omega_+:=\rho(A^{(1)}_{B_1})\cap\rho(A^{(2)}_{B_2})\cap
 \bC_+\neq\emptyset,
  \end{equation*}
where $A^{(j)}_{B_j} (\in Ext_{A^{(j)}})$ is given by \eqref{2.4}
with the boundary operator $B_j,\  j\in \{1,2\}.$ 
Then the equality
  \begin{equation}\label{5.4A}
K_1^*(B_1-M_1(z))^{-1}K_1=K_2^*(B_2-M_2(z))^{-1}K_2, \quad z\in \Omega_+,
   \end{equation}
yields   the  identities \eqref{5.5.1} and \eqref{5.5}, where the operator
measure $\wt\Sigma_j(\cd):\cB_b(\bR)\to [\cH]$  and the operators  $ \wt
C_j,\;\wt D_j \;\wt B_j\in [\cH]$ are defined by \eqref{5.5A} and \eqref{5.5B}.
    \end{proposition}
Combining Corollary \ref{cor5.0} with Proposition \ref{pr2.4.1} gives the
following result.
   \begin{proposition}\label{pr5.2}
Let the conditions of Proposition \ref{pr5.1} be satisfied, let $A_0^{(j)}=\ker
\G_0^{(j)}$, and let $E_j(\cd)$ be the spectral measure of $A_0^{(j)}, \;
j\in\{1,2\}$. Assume, in addition, that at least one of the following
assumptions is fulfilled:

(a1) Equality \eqref{5.4A} holds,  $\mul A_0^{(j)}\neq\gH_j, \; j\in \{1,2\},$
and  the spectral measure $E_j^{ac}(\cd),\ j\in \{1,2\},$ is not equivalent to
the Lebesgue measure $m_\cH(\cd)$.

(a2) The set
\begin {equation*}
 \Omega_-:=\rho(A^{(1)}_{B_1})\cap \rho(A^{(2)}_{B_2}) \cap \bC_-
\end{equation*}
is not empty and the following equality holds
  \begin {equation*}
K_1^*(B_1-M_1(z))^{-1}K_1=K_2^*(B_2-M_2(z))^{-1}K_2, \quad
z\in\Omega:=\Omega_+\cup\Omega_-.
   \end{equation*}
 Then the  relations \eqref{5.16} and \eqref{5.16a} are satisfied.
   \end{proposition}
   \begin{corollary}\label{cor5.2a}
Let the conditions of Proposition \ref{pr5.1} be satisfied, let $A_0^{(j)}=\ker
\G_0^{(j)}$, $j\in\{1,2\}$, and let the following assumption be fulfilled:

(a3) Equality \eqref{5.4A} holds,  $\mul A_0^{(j)}\neq\gH_j,$ and
$\s_{ac}(A_0^{(j)})\neq \bR,\; j\in\{1,2\}$.

Then the  relations \eqref{5.16} and \eqref{5.16a} are valid.
  \end{corollary}
The following statement is immediate from Corollary \ref{cor5.0}, (iii).
   \begin{proposition}\label{pr5.2b}
Let the conditions of Proposition \ref{pr5.1} be satisfied. Assume in addition
that the following assumption is fulfilled:

(a4) Equality \eqref{5.4A} holds and for some $t_0\in\bR$ the weak derivatives
$w\text{-}\frac{d\Sigma_1}{dt}(t_0)$, $w\text{-}\frac{d\Sigma_2}{dt}(t_0)$
exist and are equal. Then the relations \eqref{5.16} and \eqref{5.16a} are
valid.
  \end{proposition}

From Propositions \ref{pr5.2} and \ref{pr5.2b}   one can derive some general
sufficient conditions on the unitary equivalence of certain proper extensions.
However, we can deduce a more general result involving auxiliary boundary
triplets. For this purpose the following simple lemma, which is immediate from
\cite{DM95}, is needed.
   \begin{lemma}\label{lem5.3}
Let $\Pi=\{\cK,\G_0,\G_1\}$ be a boundary triplet for $A^*$ and let $M(\cd)$ be
the corresponding Weyl function. Moreover, let $\cH $ be a Hilbert space and
let $B\in [\cK]$ and $K\in [\cH,\cK]$ with $0\in\rho (K)$. Then the transform
   \begin{equation}\label{hatbt}
\wh\Gamma_0=K^* \Gamma_0,\quad
 \wh\Gamma_1=K^{-1} (\Gamma_1-(\re\, B)\Gamma_0),
   \end{equation}
defines an ordinary boundary triplet $\wh\Pi_{K,B}:= \{\cH,\wh\G_0,\wh \G_1\}$
for $A^*$ such that $\wh A_0(= \ker \wh\G_0 )=A_0$, and the corresponding Weyl
function is given by
  \begin{equation}\label{hatweyl}
\wh M(\cdot) = K^{-1}(M(\cdot) - \re \,B)K^{-1*}.
  \end{equation}
Moreover, if $T\in [\cK]$ is a boundary operator (coordinate) of $\wt A\in\exa$
in the triplet $\Pi$ (i.e. $\wt A=A_T$, see \eqref{2.4}), then the boundary
operator  of $\wt A$ in the triplet $\wh \Pi_{K,B}$ is
\begin {equation}\label{5.20}
 \wh T=K^{-1}(T-\re \,B)K^{-1*}.
\end{equation}
  \end{lemma}

Now we are ready to state the main  result regarding the unitary equivalence of
the auxiliary  boundary triplets defined by \eqref{hatbt}.
  \begin{theorem}\label{th5.4}
Let $A^{(j)}$ be a simple symmetric operator in the separable Hilbert space
$\gH_j$ and let $\Pi_j=\{\cH_j,\G_0^{(j)},\G_1^{(j)}\}$ be a boundary triplet
for $A^{(j)*}$ with the Weyl function $M_j(\cd)$, $j\in \{1,2\}$. Moreover, let
$A_{0}^{(j)}=\ker \G_0^{(j)},$ let $\cH$ be a separable Hilbert space, and let
$B_j\in [\cH_j]$ and $K_j\in [\cH,\cH_j],\; j\in \{1,2\},$ be operators such
that $0\in\rho(K_1)\cap\rho(K_2)$ and
 \begin {equation*}
 \Omega_+:=\rho(A^{(1)}_{B_1})\cap\rho(A^{(2)}_{B_2})\cap \bC_+\neq\emptyset.
\end{equation*}
Assume, in addition, that at least one of the assumptions (a1)--(a4) listed in
Propositions \ref{pr5.2}, \ref{pr5.2b} and Corollary \ref{cor5.2a} is
satisfied.  Then the boundary triplets $\wh\Pi_{K_1,B_1}$ and
$\wh\Pi_{K_2,B_2}$ defined in Lemma \ref{lem5.3} are unitarily equivalent by
means of a unitary operator $U\in [\gH_1,\gH_2]$.  In particular, the pairs of
extensions $\{A_{0}^{(1)}, A^{(1)}_{B_1}\}$ and $\{A_0^{(2)}, A^{(2)}_{B_2}\}$
are unitarily equivalent by means of the same $U$.
   \end{theorem}
   \begin{proof}
 According to Propositions  \ref{pr5.2}, \ref{pr5.2b} and Corollary \ref{cor5.2a}
$\im \wt B_1=\im \wt B_2$ and by the last equality in \eqref{5.5B} one has
    \begin{equation}\label{5.21}
K_1^{-1}(\im B_1) K_1^{-1*}=K_2^{-1}(\im B_2) K_2^{-1*}.
  \end{equation}
Combining this identity with \eqref{5.4A} one obtains
   \begin {equation}\label{5.22}
K_1^{-1}(M_1(z)-\emph{\re} B_1)K_1^{-1*} = K_2^{-1}(M_2(z)-\emph{\re}
B_2)K_2^{-1*},\quad z\in\Omega_+.
  \end{equation}

On the other hand, according to \eqref{hatweyl} the Weyl function $\wh
M_j(\cd)$ corresponding to the boundary triplet $\wh\Pi_{K_j,B_j}$ is
\begin{equation*}
\wh  M_j(z)=K_j^{-1}(M_j(z)-\re B_j)K_j^{-1*}, \quad j\in\{1,2\}, \quad
z\in\bC_+\cup \bC_-.
\end{equation*}
Combining this equality with \eqref{5.22} yields $\wh M_1(z) = \wh M_2(z)$,
$z\in\Omega_+ $.  By Theorem \ref{th2.8}, the boundary triplets
$\wh\Pi_{K_1,B_1}$ and $\wh\Pi_{K_2,B_2}$  are unitarily equivalent by means of
a  unitary operator $U\in [\gH_1,\gH_2]$.

Next, according to \eqref{5.20} a boundary operator (coordinate) $\wh B_j$ of
the extension $\wt A_j:=A_{B_j}$ in the boundary triplet $\wh\Pi_{K_j,B_j}$ is
\begin{equation*}
\wh B_j=K_j^{-1}(B_j-\re B_j)K_j^{-1*}=i\, K_j^{-1} (\im B_j)K_j^{-1*}, \quad
j\in\{1,2\}.
  \end{equation*}
 Hence, by \eqref{5.21},  $\wh B_1 = \wh B_2$ and according to  Lemma
\ref{lem2.7} the extensions $A_{B_1}$ and $A_{B_2}$ are unitarily equivalent by
means of $U$. Moreover, by  Lemma \ref{lem2.7}, the extensions $A_0^{(1)}$ and
$A_0^{(2)}$ are also unitarily equivalent by means of $U$, since their boundary
relations (coordinates) in the triplets  $\wh\Pi_{K_1,B_1}$ and
$\wh\Pi_{K_2,B_2}$ coincide: $A_{0}^{(1)} = \ker \G_0^{(1)}=\ker \wh\G_0^{(1)}$
and $A_{0}^{(2)} = \ker \G_0^{(2)} = \ker \wh\G_0^{(2)}.$
         \end{proof}

Next it is shown that the function $\wt M(\cdot)$ defined in \eqref{5.1} is the
Weyl function of a boundary triplet $\wt\Pi$ for the dual pair $\DA$. This
statement explains appearance of the function $\wt M(\cdot)$ in this section
and, in fact, has motivated our investigations. The following statement, which
is immediate from \cite{MalMog02}, establishes a key connection between Theorem
\ref{th5.4} and the theory of boundary triplets for dual pairs.
       \begin{proposition}\label{pr5.4a}
Let $A$ be a symmetric operator in $\gH$ and let $\Pi=\{\cK, \G_0, \G_1\}$ be a
boundary triplet for $A^*$ with the Weyl function $M(\cd)$. Let $B\in [\cK],\;
K\in [\cH,\cK],\; 0\in \rho (K)$, and define the linear mappings
$$
 \wt\G=\begin{pmatrix}\wt \G_0\cr\wt\G_1\end{pmatrix}:A^*\to\cH\oplus \cH,
 \quad   \wt \G^\top=\begin{pmatrix}\wt \G_0^\top\cr
\wt \G_1^\top\end{pmatrix}:A^*\to \cH\oplus \cH
 $$
by
\begin {equation*}
\wt \G_0=K^{-1}(B\G_0-\G_1), \quad \wt\G_1=K^*\G_0;\quad \wt
\G_0^\top=K^{-1}(B^*\G_0-\G_1), \quad \wt \G_1^\top=K^*\G_0.
\end{equation*}
Then $\wt\Pi=\{\cH\oplus \cH,\wt\G, \wt \G^\top\}$ is a boundary triplet for
the dual pair $\{A,A\}$, such that
\begin {equation}\label{5.23.2}
\wt A_0 :=\ker\wt \G_0 = \ker (\G_1-B\G_0),
\end{equation}
and the corresponding Weyl function is
\begin {equation}\label{5.23.3}
 \wt M(z):=M_{\wt\Pi}(z)=K^*(B-M(z))^{-1}K, \quad z\in\rho (\wt A_0).
\end{equation}
If, in addition, $B$ is accumulative, $\im B\leq 0$, then $\bC_+\subset \rho
(\wt A_0) $ and $\im \wt M(z)\geq 0 , \; z\in\bC_+.$ Hence $\wt M(\cdot)$
admits integral representation \eqref{2.2} in  $\bC_+$.
    \end{proposition}
%
%
\begin{remark}\label{rem5.4a}
Proposition \ref{pr5.4a} allows one to reformulate the statement of Theorem
\ref{th5.4}  on  the unitary equivalence of the extensions $A^{(1)}_{B_1}$ and
$A^{(2)}_{B_2}$  in terms of the equality of the Weyl functions corresponding
to special boundary triplets for the dual pair $\DA.$  The corresponding
routine reformulations are left for the reader.
  \end{remark}
        \begin{remark}\label{rem5.4b}
Let $\Pi=\bt$ be an ordinary boundary triplet for $A^*$ with the Weyl function
$M(\cd)$. It follows from the results in \cite{MalMog02} that the Weyl
functions of any boundary triplet $\wt \Pi =\{\cH\oplus \cH,\wt\G, \wt
\G^\top\}$ for the dual pair $\DA$ is obtained via the linear fractional
transform
   \begin {equation}\label{5.23.5}
\wt M(z)=(X_3+X_4 M(z))(X_1+X_2 M(z))^{-1}, \quad z\in\rho (\wt A_0),
  \end{equation}
where
  \begin {equation}\label{5.23.4}
X=\begin{pmatrix} X_1 & X_2 \cr X_3 & X_4
\end{pmatrix}: \cH\oplus\cH \to \cH\oplus\cH
  \end{equation}
is a bounded operator with bounded inverse.

If, in addition, $0\in\rho (X_2)$ and $X_3=X_4X_2^{-1}X_1-X_2^{-1*}$, then the
equality \eqref{5.23.5} takes the form
   \begin {equation}\label{5.23.6}
\wt M(z)= C + K^*(B-M(z))^{-1}K, \quad z\in\rho (\wt A_0).
   \end{equation}
with $K=-X_2^{-1}$, $B=-X_2^{-1}X_1$, and $C=X_4X_2^{-1}$.

We have no general analog of Theorem \ref{th5.4} for Weyl functions of the more
general form \eqref{5.23.6} or \eqref{5.23.5}. Moreover, we have no inner
characterization of the Weyl functions of the form \eqref{5.23.3} or
\eqref{5.23.6}.
    \end{remark}

\section{Unitary equivalent  boundary triplets for  symmetric dual pairs}
In this section we investigate unitary equivalence of some boundary triplets
for the dual pair $\DA$ in terms of the corresponding Weyl functions.

 Assume  that $A\in [\cD,\gH]$ is a bounded
symmetric operator with the closed domain $\cD\subset \gH$ and let
   \begin {equation}\label{2.14}
\gN=\gH\ominus\cD (=\mul A^*), \quad \wh \gN=\{\{0,n\}:n\in\gN\}(\subset A^*)
  \end{equation}
Consider the block-matrix  representation of $A,$ 

$$
A=\begin{pmatrix} A_{00} \cr A_{10} \end{pmatrix}: \cD\to \gH= \cD\oplus\gN.
$$

A proper extension $\wt A\in\exa$ is called bounded if $\wt A\in [\gH]$. Every
bounded extension $\wt A\in\exa$ admits a block-matrix representation
\begin{equation}\label{2.15}
\wt A=\begin{pmatrix} A_{00} & A_{10}^*\cr A_{10} &B \end{pmatrix}:
\cD\oplus\gN\to \cD\oplus\gN
\end{equation}
with some $B\in [\gN]$ and vise versa. The adjoint $A^*$ of $A$ is multivalued
and with any bounded extension $\wt A$ it admits the graph decomposition
\begin{equation}\label{A*}
  A^*=\wt A \,\dot{+}\, \widehat\gN,
\end{equation}
where $\dot{+}$ stands for the direct sum of the graphs; cf. \cite[Lemma
5.2]{MalMog05}.

  \begin{proposition}\label{pr2.14}$\,$\cite{MalMog05}
Let $\AD$ be a bounded symmetric operator, let $\Pi=\bta$ be boundary triplet
for $\DA$ such that $A_0(=\ker \G_0)$ is a bounded extension of $A$ and let
$\pi_2$ be the orthoprojector in $\gH\oplus\gH$ onto $\{0\}\oplus\gH$. Then:

(i) The  operators $\G_0\up \wh\gN$ and $\G_0^\top\up \wh\gN$ isomorphically
map $\wh\gN$ onto $\cH$, so that the operators
  \begin{gather}
\g_\Pi=\pi_2(\G_0\up \wh\gN)^{-1}, \quad \g_{\Pi_\top}=\pi_2(\G_0^\top\up
\wh\gN)^{-1},\label{2.16}\\
\G_1\up \wh\gN=\cF_\Pi\G_0\up \wh\gN\label{2.16a}
  \end{gather}
are well defined and  $\g_\Pi \in [\cH,\gN], \; \g_{\Pi_\top} \in [\cH,\gN]$,
and $\cF_\Pi\in [\cH].$

(ii) The corresponding Weyl function $M_\Pi(\cd)$ is given by
\begin {equation}\label{2.17}
 M_{\Pi}(z)=\cF_\Pi+\g_{\Pi_\top}^*(A_0-z)^{-1}\g_\Pi,
 \quad z\in\rho(A_0).
\end{equation}
\end{proposition}
Recall (\cite{DM95, Mal92, MalMog02, MalMog05}) that  the operator $\cF_\Pi$ in
\eqref{2.16a}  is called a forbidden operator corresponding to $\Pi.$

Next assume that $\AD$ is a bounded symmetric operator with the closed domain
$\cD\subset \gH$ and let $\gN$ and $\wh\gN$ be the subspaces defined in
\eqref{2.14}.
  \begin{definition}\label{def5.7}
A boundary triplet $\Pi=\bta$ for $\DA$ is said to belong to the class
$BT_\infty $ if $A_0(=\ker \G_0)\in [\gH]$ and $\G_0\{0,n\}= \G_0^\top\{0,n\},
\; n\in \gN $; in view of \eqref{2.16} this latter condition is equivalent to
the equality
\begin {equation}\label{5.31}
\g_\Pi=\g_{\Pi_\top}=:\g.
\end{equation}
  \end{definition}

Next it is shown that for each boundary triplet $\Pi=\bta$ from the class
$BT_\infty $ the linear mappings  $\G$ and $\G^\top$ can explicitly be
expressed by means of the operators $A_0,\g_\Pi,\g_{\Pi_\top}$, and $\cF_\Pi$
defined in Proposition~\ref{pr2.14}.
\begin{lemma}\label{lem5.8}
Let $A\in [\cD,\gH]$ be a bounded symmetric operator in $\gH$ and let $\Pi=\bta
\in BT_\infty  $ be a boundary triplet for $\DA$. Moreover, let $A_0=\ker
\G_0$, let $\cF_\Pi\in [\cH]$ be the forbidden operator \eqref{2.16a} and let
$\g \in [\cH, \gN]$ be the operator defined by \eqref{5.31}. Then the operators
$\G_j$ and $\G_j^\top$ admit the representations
\begin{gather}
\G_0\{f,f'\}=\g^{-1}(f'-A_0f), \quad \G_1\{f,f'\}=-\g^*P_\gN f+\cF_\Pi
\g^{-1}(f'-A_0f), \label{5.32}\\
\G_0^\top \{f,f'\}=\g^{-1}(f'-A_0^*f), \quad \G_1^\top\{f,f'\}=-\g^*P_\gN
f+\cF_\Pi^* \g^{-1}(f'-A_0^* f), \label{5.33}
\end{gather}
where $\{f,f'\}\in A^*$.

Conversely, let $\gH$ and $\cH$ be Hilbert spaces, let $A\in [\cD,\gH]\;
(\cD\subset \gH)$ be a bounded symmetric operator, let $\wt A\in [\gH]$ be a
bounded extension of $A$, let $K$ be an isomorphism from $\cH$ onto
$\gN(=\gH\ominus \cD)$, and let $F\in [\cH]$. Then the operators
$\G=\begin{pmatrix}\G_0\cr \G_1
\end{pmatrix}$ and $\G^\top=\begin{pmatrix}\G_0^\top\cr \G_1^\top
\end{pmatrix}$ defined for all $\{f,f'\}\in A^*$ by
\begin{gather}
\G_0\{f,f'\}=K^{-1}(f'-\wt A f), \quad \G_1\{f,f'\}=-K^*P_\gN f+F
K^{-1}(f'- \wt A f), \label{5.34}\\
\G_0^\top \{f,f'\}=K^{-1}(f'-\wt A^*f), \quad \G_1^\top\{f,f'\}=-K^*P_\gN f+F^*
K^{-1}(f'-\wt A^* f) \label{5.35}
\end{gather}
form the boundary triplet $\Pi=\bta\in BT_\infty$ for $\DA$. Moreover, 
\begin{equation}\label{5.35a}
A_0=\wt A, \quad \g=K, \quad \cF_\Pi=F.
\end{equation}
\end{lemma}
   \begin{proof}
Applying the identity \eqref{2.11} to the elements $\{f,A_0f\},\{0,\g h \}\in
A^*$, see \eqref{A*}, \eqref{2.16}, one obtains
\begin {equation*}
-(f,\g h)=(\G_1 \{f,A_0f\},h), \quad f\in\gH, \;\; h\in\cH.
\end{equation*}
Hence,
\begin {equation}\label{5.36}
\G_1\{f, A_0f\}=-\g^* P_\gN f, \quad f\in\gH.
\end{equation}
Moreover, \eqref{A*} shows that every $\{f,f'\}\in A^*$ can be uniquely
decomposed as
\begin{equation}\label{5.38}
\{f,f'\}=\{f,A_0f\}+\{0,n\}
\end{equation}
with $n=f'-A_0f\in \gN$. Hence, by applying $\G_0$ and $\G_1$ to the equality
\eqref{5.38} and taking \eqref{5.36}, \eqref{2.16} and \eqref{2.16a} into
account one obtains
\begin{gather*}
\G_0 \{f,f'\}=\G_0 \{f,A_0f\} +\G_0\{0,n\} =0+ \g_\Pi^{-1}n=\g^{-1}n=\g^{-1}
(f'-A_0f), \\
\G_1 \{f,f'\}=\G_1 \{f,A_0f\} +\G_1\{0,n\} =-\g^* P_\gN f+\cF_\Pi \g_\Pi^{-1}n
\\ =-\g^* P_\gN f +\cF_\Pi\g^{-1} (f'-A_0f), \;\;\; \{f,f'\}\in A^*.
\end{gather*}
Thus \eqref{5.32} is valid. Moreover, the equalities \eqref{5.33} hold, since
they are analogs of \eqref{5.32} for the triplet $\Pi_\top$.

Conversely, let the operators $\G_j$ and $\G_j^\top, \; j\in \{0,1\},$ be
defined by  \eqref{5.34} and \eqref{5.35}. Then it is immediately checked that
$\G A^*=\G^\top A^*=\cH\oplus\cH$ and the identity \eqref{2.11} is satisfied.
Hence $\Pi=\bta$ with operators \eqref{5.34} and \eqref{5.35} is a boundary
triplet for $\DA$. Moreover, the equalities \eqref{5.35a} are implied by
\eqref{5.34} and \eqref{5.35}.
  \end{proof}

Next a class of boundary triplets analogous to that appearing in
Definition~\ref{def5.7} is introduced with a finite real point (instead of
$\infty$). Here $A$ is a, not necessarily bounded or densely defined, symmetric
operator in $\gH$ such that
$\l_0\in\wh\rho(A)\cap \Bbb R\not = \emptyset$. 
\begin{definition}\label{def5.6}
A boundary triplet $\Pi=\bta$ for $\DA$ is said to belong to the class
$BT_{\l_0}$, $\l_0\in {\Bbb R}$, if $\l_0\in\rho (A_0)$ and
 \begin {equation}\label{5.30}
\G_0\{f,\l_0 f\}=\G_0^\top\{f,\l_0 f\}, \quad
f\in\gN_{\l_0}\,(=\ker(A^*-\l_0)).
\end{equation}
\end{definition}
The next lemma provides a connection between the classes $BT_{\l_0}$ and
$BT_\infty$.
   \begin{lemma}\label{lem5.9}
Let $A$ be a symmetric operator in $\gH$ with $\l_0=\ov\l_0\in \wh\rho (A)$,
let $\Pi=\bta\in BT_{\l_0}$ be a boundary triplet for $\DA$, let $A_0=\ker\G_0$
and let $M_\Pi(\cd)$ be the corresponding Weyl function. Assume also that $Y$
is an isomorphism in $\gH\oplus\gH$ given by
\begin {equation}\label{5.39.1}
Y\{f,f'\}=\{f',f+\l_0 f'\}, \quad \{f,f'\}\in \gH\oplus\gH.
\end{equation}
Then:
\begin{enumerate}\def\labelenumi{\rm (\roman{enumi})}
\item $S:=(A-\l_0)^{-1}$ is a bounded symmetric operator with the
closed domain $\dom S=\ran (A-\l_0)^{-1}$;

\item $Y(S^*)=A^*$ and the triplet $\dot \Pi
=\{\cH\oplus\cH, \dot\G,\dot\G^\top\}$ with the mappings
$$
\dot\G=\begin{pmatrix}\dot\G_0\cr \dot\G_1 \end{pmatrix}:S^*\to \cH\oplus \cH
\;\;\;\text{and}\;\;\;\dot\G^\top=\begin{pmatrix}\dot\G_0^\top\cr \dot\G_1^\top
\end{pmatrix}:S^*\to \cH\oplus \cH
$$
given by
\begin {equation}\label{5.39.2}
\dot\G_0=\G_0 Y\up S^*, \quad \dot\G_1=-\G_1 Y\up S^*, \quad
\dot\G_0^\top=\G_0^\top Y\up S^*, \quad \dot\G_1^\top=-\G_1^\top Y\up S^*
\end{equation}
is a boundary triplet for $\{S,S\}$ which belongs to the class $BT_\infty$;

\item $S_0(=\ker \dot \G_0)=(A_0-\l_0)^{-1}$ and the corresponding
Weyl function is
\begin {equation}\label{5.39.3}
M_{\dot\Pi}(z)=-M_\Pi (\l_0+\tfrac 1 z), \quad z\in\rho (S_0), \quad z\neq 0.
\end{equation}
\end{enumerate}
     \end{lemma}
\begin{proof}
(i) The statement is implied by the inclusion $\l_0\in \wh\rho (A)$.

(ii) It is clear from \eqref{5.39.1} that
\begin {equation*}
Y^{-1}\{g,g'\}=\{g'-\l_0 g,g\}, \quad \{g,g'\}\in\gH\oplus\gH.
\end{equation*}
This together with \eqref{5.39.1} gives the following identities for each
relation $T\in\C (\gH)$:
\[
 Y(T)=T^{-1}+\l_0 I , \quad 
 Y^{-1}(T)=(T-\l_0 I)^{-1}. 
\]
In particular, $Y(S^*)=S^{-1*}+\l_0 I=A^*$ and, therefore, the operators in
\eqref{5.39.2} are well defined. Now, it is straightforward to check that the
Green's identity \eqref{2.11} for the triplet $\Pi$ yields the same identity
\eqref{2.11} for the operators defined in \eqref{5.39.2}. Moreover, the
operators $\dot\G$ and $\dot \G^\top$ are surjective, because so are the
operators $\dot\G$ and $\dot \G^\top$, and $Y(S^*)=A^*$. This shows that $\dot
\Pi$ is a boundary triplet for $\{S,S\}$.

On the other hand, from the first equality in \eqref{5.39.2} one obtains
$$
S_0:=\ker \dot\G_0=Y^{-1}(A_0)=(A_0-\l_0)^{-1},
$$
so that $S_0\in [\gH]$. Moreover, for each $n\in\gN_S(=\gH\ominus\dom S)$ one
has $Y \{0,n\}=\{n,\l_0 n\}$ and hence
$$ \dot\G_0\{0,n\}=\G_0 \{n,\l_0 n\}=\G_0^\top \{n,\l_0
n\}=\dot\G_0^\top \{0, n\}.
$$
Therefore, the triplet $\dot\Pi$ belongs to the class $BT_\infty$.

(iii) If $z\neq 0, \; z\in\rho (S)$ and $f\in\gN_z (S)$, then by \eqref{5.39.2}
\begin{gather*}
\dot\G_0 \{f,zf\}=\G_0 Y\{f,zf\}=z\G_0\{f,(\l_0+\tfrac 1 z)f\},\\
 \dot\G_1\{f,zf\}=-z \G_1\{f,(\l_0+\tfrac 1 z)f\}.
\end{gather*}
Combining these equalities with the definition \eqref{2.12} of the Weyl
function yields  formula \eqref{5.39.3} for $M_{\dot\Pi}(\cdot)$.
\end{proof}

Now we are ready to prove the main theorems in this section which improve
Theorem \ref{th5.4} for the case of operators $A_0^{(j)}$ with a real regular
point.
\begin{theorem}\label{th5.10}
Let $\gH_j$ be a Hilbert space, let $A^{(j)}\in [\cD_j,\gH_j]$ be a bounded
simple symmetric operator with the closed domain $\cD_j\subset \gH_j$ and let
$\Pi_j=\{\cH\oplus\cH,\G^{(j)},\G^{\top (j)} \}\in BT_\infty$ be a boundary
triplet for $\{A^{(j)},A^{(j)}\}$ with the mappings
$$
\G^{(j)}=\begin{pmatrix} \G_0^{(j)}\cr \G_1^{(j)}\end{pmatrix}:A^{(j)*}\to
\cH\oplus\cH, \quad  \G^{\top (j)}=\begin{pmatrix} \G_0^{\top(j)}\cr \G_1^{\top
(j)}\end{pmatrix}:A^{(j)*}\to \cH\oplus\cH.
$$
Moreover, let $A^{(j)}_0 (=\ker \G_0^{(j)})$ be a bounded non-self-adjoint
operator  and let $M_{\Pi_j}(\cd)$ be the corresponding Weyl function, $j\in
\{1,2\}$. If for some $R>0$,
  \begin {equation}\label{5.40}
M_{\Pi_1}(z)=M_{\Pi_2}(z), \quad |z|>R,
 \end{equation}
then the boundary triplets $\Pi_1$ and $\Pi_2$ are unitarily equivalent.
 \end{theorem}
\begin{proof}
Let $\gN_j=\gH_j\ominus \cD_j$ and let $\g_{\Pi_j}\in [\cH,\gN_j]$,
$\g_{\Pi_\top,j}\in [\cH,\gN_j]$ and $\cF_{\Pi_j}\in [\cH]$ be the operators in
\eqref{2.16} corresponding to the triplet $\Pi_j$, $j\in \{1,2\}$. Then
according to Definition \ref{def5.7} $\g_{\Pi_j}=\g_{\Pi_\top,j}=:\g_j$ and
hence \eqref{2.17} yields
\begin {equation}\label{5.41}
M_{\Pi_j}(z)=\cF_{\Pi_j}+\g_j^* (A^{(j)}_0-z)^{-1}\g_j, \quad z\in \rho
(A^{(j)}_0), \;\; j\in \{1,2\}.
\end{equation}
Clearly, $\cF_{\Pi_j}=\lim\limits_{z\to \infty}M_{\Pi_j}(z)$ and now it follows
from \eqref{5.40} that
\begin{gather}
\cF_{\Pi_1}=\cF_{\Pi_2}=:\cF,\label{5.42}\\
\g_1^*(P_{\gN_1} (A^{(1)}_0-z)^{-1}\up\gN_1)\g_1=\g_2^*(P_{\gN_2}
(A^{(2)}_0-z)^{-1}\up\gN_2)\g_2, \quad |z|>R.\label{5.43}
\end{gather}

Let $\wt A^{(j)}_0 \in [\gH_j]$ be a self-adjoint extension of  $A^{(j)}$
defined by
\begin {equation}\label{5.44}
\wt A^{(j)}_0=\begin{pmatrix} A_{00}^{(j)} &  A_{10}^{(j)*} \cr A_{10}^{(j)} &
0 \end{pmatrix}:\cD_j\oplus\gN_j\to \cD_j\oplus\gN_j, \quad j\in \{1,2\},
\end{equation}
(i.e., by \eqref{2.15} with $B=0$) and define the mappings $\wt \G_0^{(j)},\;
\wt \G_1^{(j)}:A^{(j)*}\to \cH, \;\;j\in \{1,2\}, $ by
\begin {equation}\label{5.45}
\wt \G_0^{(j)}\{f,f'\}=\g_j^* P_{\gN_j}f, \quad \wt
\G_1^{(j)}\{f,f'\}=\g_j^{-1}(f'- \wt A^{(j)}_0 f), \quad \{f,f'\}\in A^{(j)*}.
\end{equation}
It follows from \cite[Proposition 3.5]{DM95} that the collection
$\wt\Pi_j=\{\cH,\wt \G_0^{(j)}, \wt \G_1^{(j)}\}$ is an ordinary boundary
triplet for $A^{(j)*}$ with the Weyl function
\begin {equation}\label{5.46}
\wt M_j(z)=\g_j^{-1}(z I + A_{10}^{(j)}(A_{00}^{(j)}-z)^{-1}A_{10}^{(j)*})
\g_j^{-1*}, \quad z \in \rho (A_{00}^{(j)}), \quad j\in \{1,2\}.
\end{equation}

Next, in view of \eqref{2.15} the operator $ A_{0}^{(j)}$ has the block
representation
\begin {equation}\label{5.47}
 A^{(j)}_0=\begin{pmatrix} A_{00}^{(j)} &  A_{10}^{(j)*} \cr A_{10}^{(j)} &
B_j \end{pmatrix}:\cD_j\oplus\gN_j\to \cD_j\oplus\gN_j, \quad j\in \{1,2\}
\end{equation}
with some $B_j\in [\gN_j]$. Applying the Frobenius formula to \eqref{5.47} and
taking \eqref{5.46} into account  one gets
\begin {gather*}
P_{\gN_j} (A^{(j)}_0-z)^{-1}\up\gN_j=\left (B_j-z I - A_{10}^{(j)}
(A_{00}^{(j)}-z)^{-1}A_{10}^{(j)*} \right  )^{-1}  \\
=  \left( B_j - \g_j\wt M_j(z)\g_j^* \right )^{-1}=\g_j^{-1*} \left (\g_j^{-1}
B_j \g_j^{-1*} - \wt M_j(z)\right )^{-1}\g_j^{-1}, \quad j\in \{1,2\}.
\end{gather*}
Substituting these identities into \eqref{5.43} yields the equality
\begin{equation}\label{5.48}
\wt B_1 - \wt M_1 (z)=\wt B_2 - \wt M_2 (z), \quad |z|>R,
\end{equation}
where $\wt B_j=\g_j^{-1} B_j \g_j^{-1*}(\in [\cH]) $. It follows easily from
\eqref{5.48} and \eqref{5.46} that
\begin{gather}
\wt B_1=\wt B_2=:\wt B, \label{5.49}\\
\wt M_1 (z)=\wt M_2 (z), \quad z\in\bC\setminus \bR, \label{5.50}
\end{gather}
(cf. Corollary~\ref{cor5.0}(ii)). On the other hand, the equality \eqref{5.32}
together with \eqref{5.44} and \eqref{5.47} gives
\begin {gather*}
\G_0^{(j)}\{f,f'\}= \g_j^{-1}(f'-  A^{(j)}_0 f)=\g_j^{-1}(f'- \wt A^{(j)}_0 f)
-\g_j^{-1} B_j P_{\gN_j}f \\
 =\g_j^{-1}(f'- \wt A^{(j)}_0 f)- \wt B_j \g_j^*P_{\gN_j}f,\\
\G_1^{(j)}\{f,f'\}=-\g_j^*P_{\gN_j}f+ \cF_{\Pi_j}\g_j^{-1}(f'- \wt A^{(j)}_0
f)-\cF_{\Pi_j}\wt B_j \g_j^*P_{\gN_j}f \\
 =(-I_\cH-\cF_{\Pi_j}\wt B_j)\g_j^*P_{\gN_j}f+\cF_{\Pi_j}\g_j^{-1}(f'- \wt
A^{(j)}_0 f).
\end{gather*}
Comparing these equalities with \eqref{5.45} and taking \eqref{5.42} and
\eqref{5.49} into account we obtain for $j\in \{1,2\}$ and all $\{f,f'\}\in
A^{(j)*}$,
\begin {gather*}
\G_0^{(j)}\{f,f'\}=\wt\G_1^{(j)}\{f,f'\}-\wt B\wt\G_0^{(j)}
\{f,f'\}, \\ 
\G_1^{(j)}\{f,f'\}=(-I_\cH-\cF\wt B)\wt\G_0^{(j)} \{f,f'\}+ \cF
 \wt\G_1^{(j)}\{f,f'\}. 
\end{gather*}
With $X=\begin{pmatrix} -\wt B & I_\cH \cr -I_\cH-\cF\wt B & \cF
\end{pmatrix}$ and $\wt \G^{(j)}=\begin{pmatrix}\wt \G_0^{(j)} \cr  \wt\G_1^{(j)}
\end{pmatrix}$
one can rewrite the previous two equalities in the form
  \begin {equation}\label{5.53}
\G^{(j)}=X\wt \G^{(j)}, \quad j\in \{1,2\}.
\end{equation}

It follows from \eqref{5.50} and Theorem \ref{th2.8} that the ordinary boundary
triplets $\wt\Pi_1$ and $\wt\Pi_2$ are unitarily equivalent, that is
\begin {equation*}
 \wt U(A^{(1)*})=A^{(2)*} \;\;\;\text{and} \;\;\;
 \wt \G^{(2)} \wt U\up A^{(1)*}=\wt \G^{(1)}.
\end{equation*}
with some unitary operator $\wt U\in [\gH_1,\gH_2]$. This and \eqref{5.53}
prove the unitary equivalence of the triplets $\Pi_1$ and $\Pi_2$.
\end{proof}
\begin{theorem}\label{th5.11}
Let $A^{(j)}$ be a, not necessarily densely defined, simple symmetric operator
in $\gH_j$, let $\l_0=\ov\l_0\in \wh\rho (A^{(j)})$, and let
$\Pi_j=\{\cH\oplus\cH,\G^{(j)},\G^{\top (j)} \}\in BT_{\l_0}$ be a boundary
triplet for $\{A^{(j)},A^{(j)}\}$ with the Weyl function $M_{\Pi_j}(\cd)$,
$j\in \{1,2\}$. If for some $\varepsilon>0$,
\begin {equation}\label{5.54}
M_{\Pi_1}(z)=M_{\Pi_2}(z), \quad |z-\l_0|<\varepsilon,
\end{equation}
then the boundary triplets $\Pi_1$ and $\Pi_2$ are unitarily equivalent.
\end{theorem}
\begin{proof}
Let $S^{(j)}=(A^{(j)}-\l_0)^{-1}$, let $\dot\Pi_j=\{\cH\oplus\cH, \dot\G^{(j)},
\dot\G^{\top (j)} \}\in BT_\infty$ be a boundary triplet for
$\{S^{(j)},S^{(j)}\}$ constructed in Lemma \ref{lem5.9} and let
$M_{\dot\Pi_j}(z)=-M_{\Pi_j} (\l_0+\tfrac 1 z)$ be the Weyl function for
$\dot\Pi_j$ (see \eqref{5.39.3}). Then in view of \eqref{5.54} for some $R>0$
one has
$$
M_{\dot\Pi_1}(z)=M_{\dot\Pi_2}(z), \quad |z|>R.
$$
Therefore according to Theorem \ref{th5.10} the triplets $\dot\Pi_1$ and
$\dot\Pi_2$ are unitarily equivalent,  which implies that the triplets $\Pi_1$
and $\Pi_2$ are unitarily equivalent as well.
\end{proof}
\begin{corollary}\label{cor5.12}
Let $\cH,\; \gH_1$ and $\gH_2$ be Hilbert spaces, let $\gN_j$ be a subspace of
$\gH_j$ and let $\wt A_j\in [\gH_j], \; F_j \in [\cH]$ and $K_j\in [\cH,\gH_j]$
be operators such that $\gN_j\supset \ran (\wt A_j-\wt A_j^*), \; \ker
K_j=\{0\}$ and $\ran K_j=\gN_j, \; j\in \{1,2\}$. Assume also that
\begin {equation}\label{5.56}
\ov{\rm {span}} \{\wt A_1^n\gN_1:n=0,1,\dots\}=\gH_1 \;\;\;\text{and} \;\;\;
\ov{\rm {span}} \{\wt A_2^n\gN_2:n=0,1,\dots\}=\gH_2.
\end{equation}
If under the above assumptions the equality
\begin {equation}\label{5.57}
F_1+K_1^*(\wt A_1-z)^{-1}K_1=F_2+K_2^*(\wt A_2-z)^{-1}K_2, \quad |z|>R,
\end{equation}
holds for some $R>0$, then there exists a unitary operator $U\in [\gH_1,\gH_2]$
such that
  \begin {equation}\label{5.58}
UK_1=K_2 \quad \text{and} \quad  U\wt A_1=\wt A_2 U.
  \end{equation}
\end{corollary}
    \begin{proof}
Let $\cD_j=\gH_j\ominus\gN_j, \; j\in \{1,2\}$. Since $\cD_j \subset \ker (\wt
A_j-\wt A_j^*)$, the operator $A^{(j)}:=\wt A_j\up \cD_j$ is symmetric  in
$\gH_j$ with the closed domain $\cD_j, \;j\in \{1,2\} $. Moreover, the
relations \eqref{5.56} imply the simplicity of the operators $A^{(1)}$ and
$A^{(2)}$.

Letting in \eqref{5.34} and \eqref{5.35} $K=K_j, \; \wt A=\wt A_j$ and $F=F_j$
we construct the boundary triplet $\Pi_j=\{\cH\oplus\cH, \G^{(j)},\G^{ \top(j)}
\} \in BT_\infty $ for $\{A^{(j)}, A^{(j)}\}$ such that the corresponding  Weyl
function is
   \begin {equation*}\label{5.59}
M_{\Pi_j}(z)=F_j+K_j^*(\wt A_j-z)^{-1}K_j, \quad j\in \{1,2\}
   \end{equation*}
(see Proposition \ref{pr2.14}). Since by \eqref{5.57} $M_{\Pi_1}(z)=
M_{\Pi_2}(z), \; |z|>R, $ it follows from Theorem \ref{th5.10} that the
triplets $\Pi_1$ and $\Pi_2$ are unitarily equivalent by means of a unitary
operator $U\in [\gH_1, \gH_2]$. In particular, this yields the relations in
\eqref{5.58}.
  \end{proof}

\begin{remark}\label{rem5.13}
(i) In the case that $\cH_j=\gN_j$, $F_j=0$ and $K_j=I_{\gN_j}$, $j\in
\{1,2\}$, the equality \eqref{5.57} takes the form
$$
P_{\gN_1}(\wt A_1-z)^{-1}\up \gN_1=P_{\gN_2}(\wt A_2-z)^{-1}\up \gN_2, \quad
|z|>R.
$$
For this case the statement of Corollary \ref{cor5.12} was proved in
\cite[Theorem~6.2]{ArlHasSno05}.

(ii) As it is known (see for instance \cite{AroNud02}) a linear stationary
dynamical discrete-time system (LSDS) is a collection
\begin {equation}\label{5.60}
\alpha =\{\wt A, K,N,F; \gH,\cH_0, \cH_1\}
\end{equation}
of Hilbert spaces $\gH,\;\cH_0, \; \cH_1$, and operators $\wt A \in [\gH], \;
K\in [\cH_0,\gH], \; N \in [\cH_1,\gH]$ and $F\in [\cH_0,\cH_1]$. Moreover, the
operator function
\begin {equation}\label{5.61}
\Theta_\alpha (z)=F+ N^*(\wt A -z)^{-1}K, \quad z\in\rho (\wt A),
\end{equation}
is called the transfer function of the system $\alpha$.

Two systems $\alpha_j =\{\wt A_j, K_j,N_j,F_j; \gH_j,\cH_0, \cH_1\}, \; j\in
\{1,2\}$, are called similar (unitarily similar) if there is an operator $U \in
[\gH_1, \gH_2] $ with $0\in\rho (U)$ (resp. unitary $U$) such that
$$
U\wt A_1=\wt A_2 U, \quad U K_1=K_2, \quad U N_1=N_2.
$$
Some sufficient conditions for similarity and unitary similarity of $LSDS$ with
the same transfer function has been discovered in \cite{Aro79,AroNud02} (for
systems of other types see \cite{AroNud96}). For systems involving normal main
operators $\wt A_j$ some sufficient conditions for their unitary similarity
have been obtained in \cite{ArlHasSno07}.

If in addition the system \eqref{5.60} satisfies the conditions
\begin{gather}
\cH_0=\cH_1=:\cH, \quad N=K,\quad 0\in \wh \rho (K), \quad \ran K \supset \ran
(\wt A-\wt A^*), \label{5.62}\\
\ov{\rm {span}} \{\wt A^n K \cH:n=0,1,\dots\}=\gH\label{5.63}
\end{gather}
(the relation \eqref{5.63} means that the system $\alpha$ is simple in the
sense of \cite{Aro79,AroNud02}), then the transfer function takes the form
\begin {equation}\label{5.64}
\Theta_\alpha (z)=F+ K^*(\wt A -z)^{-1}K, \quad z\in\rho (\wt A)
\end{equation}
and in view of Lemma \ref{lem5.8} and Proposition \ref{pr2.14} $\Theta_\alpha
(\cd)$ is the Weyl function corresponding to some boundary triplet $\Pi\in
BT_\infty$. Moreover, it follows from \eqref{5.64} and Corollary \ref{cor5.12}
that the transfer function defines the system \eqref{5.60} satisfying
\eqref{5.62} and \eqref{5.63} uniquely up to unitary similarity.

In \cite{ArlHasSno07} passive systems of the form \eqref{5.62} (i.e. the
associated $2\times 2$ block operator with entries $\wt A, K^*,K$, and $F$ is
contractive)
are called passive quasi-selfadjoint systems, or shortly $pqs$-systems. For
such systems Corollary \ref{cor5.12} was proved in another way in
\cite[Proposition~4.3]{ArlHasSno07}, see also \cite[Theorem~3.5]{ArlHasSno07}
for an extension of this result where normal main operators $\wt A_j$ are
allowed.
\end{remark}

\section{Negative results}

       \begin{theorem}\label{thm6B}
Let $A^{(1)}$ be a simple symmetric operator in  $\gH_1$, let
$\Pi_1=\{\cH,\G_0^{(1)},\G_1^{(1)}\}$ be an ordinary boundary triplet for
$A^{(1)*}$ and   $M_1(\cdot)$ the corresponding Weyl function. If  $B_1\in
[\cH]$ and $0\in\rho(M_1(z_0)-B_1)$ for some $z_0\in \dC_+$, then the following
statements hold:

\begin{enumerate}\def\labelenumi{\rm (\roman{enumi})}

\item    There exists a (non-unique) simple symmetric operator
$A^{(2)}$, an ordinary boundary triplet $\Pi_2=\{\cH,\G_0^{(2)},\G_1^{(2)}\}$
for $A^{(2)*}$, a (non-unique) bounded dissipative operator $B_2$, and an open
neighborhood $\Omega_+\subset \dC_+$ of $z_0$, such that the following
equality holds
   \begin{equation}\label{inveq1}
 (B_1-M_1(z))^{-1}=(B_2-M_2(z))^{-1}, \quad z\in \Omega_+.
   \end{equation}

\item  The boundary triplets $\wh\Pi_{I,B_1}$ and $\wh\Pi_{I,
B_2}$ defined in Lemma \ref{lem5.3} are not unitarily equivalent.

\item  If $B_1$ is not dissipative, the extensions
\begin {equation}\label{6.0}
 A_{B_1}^{(1)} = \ker(\G_1^{(1)}-B_1\G_0^{(1)})\;\;\;\;\text{and} \;\;\;\;
 A_{B_2}^{(2)} = \ker(\G_1^{(2)}-B_2\G_0^{(2)})
\end{equation}
are not unitarily similar.
\end{enumerate}

If in addition $B_1$ is accumulative and $B_1\neq B_1^*$, then
$0\in\rho(M_1(z_0)-B_1)$ for every  $z_0\in \dC_+$ and the statement (i) holds
true with $\Omega_+=\bC_+$. In particular, the extensions $A_{B_1}^{(1)}$ and
$A_{B_2}^{(2)}$  are not unitarily similar.

\end{theorem}
  \begin{proof}
(i)  Since $M_1(\cdot)$ belongs to $R[\cH]$ the assumption
$0\in\rho(M_1(z_0)-B_1)$ implies that there exists an open neighborhood
$\Omega_+\subset \dC_+$ of $z_0$, such that $0\in\rho(M_1(z)-B_1)$ for all
$z\in\Omega_+$. Now choose a dissipative operator $B\in[\cH]$ such that $\re
B\not = 0,$ and the sum $B_2:=B+B_1$ is also dissipative and define the
function $M_2(\cd)$ by
\begin{equation}\label{6B1}
 M_2(z):=M_1(z)+B, \quad z\in\dC_+;
 \quad
 M_2(z):=M_1(z)+B^*, \quad z\in\dC_-.
\end{equation}
Since $M_1(\cd)$ is the Weyl function corresponding to the ordinary boundary
triplet $\Pi_1,$  $M_1(\cd) \in R^u[\cH]$, i.e. $M_1(\cd) \in R[\cH]$ and
$0\in\rho(\im M(i))$. Combining this fact with the inequality  $\im B\ge 0$ we
get that $M_2(\cd) \in R^u[\cH]$. Therefore (see \cite{DM95}) there exists a
simple symmetric operator $A^{(2)}$ and an ordinary boundary triplet
$\Pi_2=\{\cH,\G_0^{(2)},\G_1^{(2)}\}$ for $A^{(2)*}$ such that the
corresponding Weyl function is equal to $M_2(\cd)$. By construction,
$B=B_2-B_1\neq 0$ is dissipative and \eqref{6B1} implies that
   \begin{equation}\label{6B1New}
M_2(z)-B_2 = M_1(z) - B_1,\qquad  z\in\dC_+.
  \end{equation}
Since  $0\in\rho(M_1(z)-B_1)$ for $z\in \Omega_+$, the last identity leads to
the inclusion   $0\in\rho(M_2(z)-B_2)$  for $z\in \Omega_+.$  Taking inverses
of both sides of \eqref{6B1New} yields  \eqref{inveq1}.

(ii)  By  Lemma  \ref{lem5.3},  a collection $\wh\Pi_{I,B_j}=
\{\cH,\wh\G_0^{(j)},\wh \G_1^{(j)}\},$  \ $j\in \{1,2\},$ with

   \begin{equation}\label{hatbt5}
\wh\Gamma_0^{(j)} =  \Gamma_0^{(j)},\quad
 \wh\Gamma_1^{(j)} = (\Gamma_1^{(j)} - (\re\,
 B_j)\Gamma_0^{(j)}),\qquad j\in \{1,2\},
   \end{equation}
defines an ordinary boundary triplet  for $A^{(j)*}$ such that the
corresponding Weyl function   is
     \begin{equation}\label{hatweyl5new}
\wh M_j(z) = M_j(z) - \re \,B_j, \qquad z\in \dC_+, \qquad j\in \{1,2\}.
  \end{equation}
According to our choice $\re B\not = 0,$ i.e.\  $\re B_1 \not = \re B_2.$ Thus,
by \eqref{hatweyl5new} $\wh M_1(\cdot)\not = \wh M_2(\cdot).$   By Theorem
\ref{th2.8}, the boundary triplets $\wh\Pi_{I,B_1}$ and $\wh\Pi_{I,B_2}$ are
not unitarily equivalent.

Note also that as it is clear from \eqref{6B1New},  $M_1(\cd)\neq M_2(\cd)$.
Therefore, by Theorem \ref{th2.8},  the boundary triplets $\Pi_1$ and $\Pi_2$
are not unitarily equivalent too.

(iii)  Since $\Pi_1$ and $\Pi_2$ are ordinary boundary triplets, Proposition
\ref{pr2.4}(ii) shows that the linear relation $A_{B_1}^{(1)}$ is not
dissipative, if so is $B_1$.  At the same time, by Proposition \ref{pr2.4}(ii),
$A_{B_2}^{(2)}$ is $m$-dissipative since, by construction, $B_2$ is
dissipative. In particular, $A_{B_1}^{(1)}$ and $A_{B_2}^{(2)}$ are not
unitarily equivalent.
%
\end{proof}

\begin{remark}\label{rem5}
\begin{enumerate}\def\labelenumi{\rm (\roman{enumi})}
\item
 Theorem~\ref{thm6B}, as well as its proof, remains valid for generalized
boundary triplets in the sense of \cite{DM95}; their unitary equivalence is
defined precisely in the same way as was done for ordinary boundary triplets in
Definition \ref{def2.6}.

\item
 To demonstrate that Theorem~\ref{thm6B} holds  for a wide class of
non-accumulative operators $B_1$, we fix $z_0\in \dC_+$ and recall that as a
Weyl function of an ordinary boundary triplet $M_1(\cd)\in R^u[\cH]$. The
latter  means that $\im (M_1(z_0)f,f)\ge \varepsilon_0 \|f\|^2$ for some
$\varepsilon_0>0$. Choose an arbitrary operator $B_1\in[\cH]$ (not necessarily
accumulative) which satisfies $\|\im B_1\|\le \varepsilon_0/2$. Then, clearly
\[
 \im ((M_1(z_0)-B_1)f,f)\ge \im (M_1(z_0)f,f)-|\im(B_1f,f)|\ge
 \frac{\varepsilon_0}{2}\, \|f\|^2.
\]
This implies that $0\in\rho(M_1(z_0)-B_1)$ and hence there exists an open
neighborhood $\Omega_+\subset \dC_+$ of $z_0$, such that $0\in\rho(M_1(z)-B_1)$
for all $z\in\Omega_+$. Again the operator $A_{B_1}$ is not dissipative, if
$B_1$ is not dissipative, and thus all the conclusions of Theorem~\ref{thm6B}
hold.

\item Let under the assumptions of Theorem \ref{thm6B} $n_\pm
(A^{(1)})=n<\infty$, so that $\dim \cH=n$. Then for any fixed $z_0\in \bC_+$
the inequality $\det (M_1(z_0)-B_1)\neq  0$ holds for almost every (with
respect to the Lebesgue measure in $\bC^{n^2}$) non-dissipative matrices
$B_1\in \Bbb C^{n\times n}$. Therefore in this  case the conclusions  of
Theorem \ref{thm6B} are valid for almost every non-dissipative matrix $B_1\in
\Bbb C^{n\times n}$.

\item\label{point4}
Reasonings of Theorem \ref{thm6B} can easily be extended to establish the
following statement:

\emph {There exist simple symmetric operators $A^{(j)}$, ordinary boundary
triplets $\Pi_j=\{\cH,\G_0^{(j)},\G_1^{(j)} \}$ for $A^{(j)*}$ with the
corresponding Weyl functions $M_j(\cd)$, $j\in \{1,2\}$, a selfadjoint operator
$B_1\in  [\cH]$ and an accumulative operator  $B_2\in  [\cH]$,
 such that}
\begin {equation*}
\wt M_1(z) := (B_1-M_1(z))^{-1}=(B_2-M_2(z))^{-1} =:\wt M_2(z), \quad z\in
\bC_+,
\end{equation*}

\noindent \emph{but the extensions \eqref{6.0} are not unitarily similar.}
   \end{enumerate}
This statement shows that even in the case of maximal accumulative extension
$\wt A_B$ the Weyl function $\wt M(\cd)$ of the form \eqref{5.23.3} does not
determine  the extension uniquely up to the unitary similarity.
\end{remark}
%
%

    \begin{example}
Let $\cH$ be a separable Hilbert space. Consider in $L^2({\Bbb R},\cH)$ the
momentum operator $A_0 = -i\frac{d}{dx},$ \ $\dom(A_0)= W^{1,2}({\Bbb R})$  and
its restriction
          \begin{equation*}
A = -i\frac{d}{dx},\quad \dom(A)=W^{1,2}_0({\Bbb R}_-)\oplus W^{1,2}_0({\Bbb
R}_+)=\{f\in W^{1,2}({\Bbb R}):f(0)=0\}.
   \end{equation*}
Clearly, $A$ is symmetric operator with equal deficiency indices $n_{\pm} =
\dim\cH.$

Define the boundary triplets $\Pi_j = \{\cH,\Gamma^{(j)}_0,\Gamma^{(j)}_1\},\
j\in\{1,2\}$, for $A^*$, by setting
    \begin{equation}\label{6.40}
\sqrt{2}\Gamma^{(1)}_0 f := [f(+0) - f(-0)], \quad \sqrt{2}\Gamma^{(1)}_1 f:=
i[f(+0) + f(-0)],
    \end{equation}
and
    \begin{equation} \label{6.41}
\Gamma^{(2)}_0 := 3^{-1/2}\Gamma^{(1)}_0, \quad
 \Gamma^{(2)}_1 := \sqrt{3}\Gamma^{(1)}_1.
         \end{equation}
It easily follows from \eqref{6.40}, \eqref{6.41} that the corresponding Weyl
functions $M_j(\cdot)$, $j\in\{1,2\},$ are
     \begin{equation}\label{6.42}
M_1(z)=
\begin{cases}
iI_{\cH}, & z\in{\Bbb C}_+,\\
-iI_{\cH},& z\in{\Bbb C}_-,
   \end{cases},\qquad
M_2(z)=
\begin{cases}
3iI_{\cH}, & z\in{\Bbb C}_+,\\
-3iI_{\cH},& z\in{\Bbb C}_-.
\end{cases}
     \end{equation}
Now, let $B_1 = -iI_{\cH}$ and $B_2 = iI_{\cH}$ and define the proper
extensions $A_{B_j}$ of $A$ by
      \begin{equation}\label{6.43}
A_{B_j} = A^*\lceil\dom A_{B_j}, \quad \dom A_{B_j} =  \ker( \Gamma_1^{(j)} -
B_j\Gamma_0^{(j)}),  \quad j\in \{1,2\}.
   \end{equation}
Then by Proposition \ref{pr2.4} $A_{B_1}$ is $m$-accumulative and $A_{B_2}$ is
$m$-dissipative, because so are $B_1$ and $B_2$. More precisely, by rewriting
$\dom A_{B_1}$ in \eqref{6.43} in the form
$\dom A_{B_1} = W^{1,2}({\Bbb R}_-)\oplus W^{1,2}_0({\Bbb R}_+),$
one obtains
   \begin{equation*}
 \ker(A_{B_1}-z)=\{\,e^{izx}\chi_-(x)h:\  h\in\cH\}, \;\; z\in{\Bbb C}_-,
 \quad \sigma_p(A_{B_1})={\Bbb C}_-
   \end{equation*}
and $\sigma(A_{B_1})= \overline{\Bbb C}_-$, where $\chi_{-}(\cdot)$ is the
indicator function of ${\Bbb R}_{-}$. Similarly, it follows from \eqref{6.43}
that
    \begin{equation}
\dom A_{B_2} = \{\,f\in W^{1,2}({\Bbb R}_-)\oplus W^{1,2}({\Bbb R}_+):\  f(+0)
= -2f(-0)\}.
    \end{equation}
Hence the functions
    \begin{equation}
f_{\lambda}(x)=
\begin{cases}
e^{i\lambda x},&x>0,\\
-2^{-1}e^{i\lambda x},&x<0,
\end{cases}
    \end{equation}
form the complete family of (generalized) eigenfunctions of the continuous
spectrum of $A_{B_2}$, and $\sigma(A_{B_2})=\sigma_c(A_{B_2})={\Bbb R}$.

Thus, the operators $A_{B_1}$ and $A_{B_2}$ are not similar. At the same time,
    \begin{equation}
\bigl(B_1-M_1(z)\bigr)^{-1} = -i/2\cdot I_{\cH} = \bigl(B_2-M_2(z)\bigr)^{-1},
\quad z\in{\Bbb C}_+.
    \end{equation}

To prove the similarity of $A_{B_2}$ to $A_0$ consider the characteristic
function $W_2(\cdot)$ of the operator $A_{B_2}$. Setting $K^*=K=I_{\cH}=J$ one
has $\im B_2=I_{\cH}=KJK^*$. Hence, using \eqref{6.42} one obtains
    \begin{eqnarray*}
& W_2(z)  = I_{\cH} + 2i K^*\bigl(B^*_2 - M_2(z)\bigr)^{-1}KJ
  =I_{\cH} + 2i\bigl(B^*_2 - M_2(z)\bigr)^{-1}\\
  & = I_{\cH} + 2i(-i I_{\cH} - 3iI_{\cH})^{-1} =1/2\cdot I_{\cH},
 \quad
z\in{\Bbb C}_+.
    \end{eqnarray*}
Since $W_2^{-1}(z) =2I_{\cH}$ is bounded in ${\Bbb C}_+$, the Nagy-Foias
theorem (\cite[Theorem 9.1.2]{NaFo67}) yields the similarity of
 $A_{B_2}$  to a self-adjoint operator.

Furthermore, it is easily seen that the operator $A_{B_2}$ is completely
non-self-adjoint. In addition,
  $$
W_2(x+i0)= s-\lim_{y \downarrow 0} W_2(x+iy) = 1/2\cdot I_{\cH},\quad x\in{\Bbb
R}.
  $$
Thus, $\|W_2(x+i0)\|<1$ for $ x\in{\Bbb R}$ and,  by \cite[Corollary
9.1.3]{NaFo67}, the  operator $A_{B_2}$ is similar to the multiplication
operator $Q:\ f(x)\to xf(x)$ in $L^2({\Bbb R},\cH)$. It remains to note that
the operator  $A_0$ is unitarily equivalent to the multiplication operator $Q$,
too.
            \end{example}
       \begin{remark}
(i)  In this example $A^{(1)}=A^{(2)}=A$ and $A^{(1)}_0=A^{(2)}_0=A_0$. It is
easily seen that the spectral measure $E_{A_0}(\cdot)$ of $A_0$ is spectrally
equivalent to the Lebesgue measure $I_{\cH}dt$. This example shows that Theorem
\ref{th5.4} is sharp and the assumptions on the spectral measures
$E_{A^{(1)}_0}$ and $E_{A^{(2)}_0}$ cannot be dropped.

(ii)  It follows from \eqref{6.42} that  $\ker(B_1-M_1(z)) = \ker(-iI_{\cH} +
iI_{\cH}) = \cH,\ z\in \dC_-,$ and $0\in\rho\bigl(B_2-M_2(z)\bigr), z\in{\Bbb
C}_{\pm}.$ Thus, by Proposition \ref{pr2.5} $ \sigma_p(A_{B_1})={\Bbb C}_-$ and
$\sigma(A_{B_2})\subset{\Bbb R}$.
      \end{remark}

\end{document}